%% file: diffOpKR.tex
\documentclass[10pt]{article}
\input{paperstyle.tex}

\title{On $\sl(N)$ link homology with mod $N$ coefficients}
\author{Joshua Wang}
\date{}

\begin{document}
\maketitle

\begin{abstract}
	We construct an operator on $\sl(N)$ link homology with coefficients in a ring whose characteristic divides $N$. If $P$ is a prime number, we use this operator to exhibit structural features of $\sl(P)$ link homology that are special to characteristic $P$ that generalize well-known features of Khovanov homology that are special to characteristic $2$. 
\end{abstract}

\section{Introduction}\label{sec:Introduction}

A peculiar aspect of Khovanov homology \cite{MR1740682}, which associates a bigraded $R$-module $\Kh(L;R)$ to an oriented link $L\subset S^3$ for any commutative ring $R$, is the frequent dichotomy in structural behavior depending on whether the characteristic of $R$ is $2$. One early instance of this phenomenon concerns the relationship between (unreduced) Khovanov homology and reduced Khovanov homology $\rKh(L,q;R)$, an invariant of $L$ equipped with a distinguished component marked by a basepoint $q \in L$. Shumakovitch \cite[Theorem 3.2.A]{MR3205577} proves that if $R = \Z/2$, then $\rKh(L,q;\Z/2)$ is independent of the choice of basepoint and there is an isomorphism \[
	\Kh(L;\Z/2) \cong \rKh(L,q;\Z/2)\otimes \frac{(\Z/2)[X]}{X^2}.
\]The basepoint $q$ induces a basepoint operator $X_q\colon \Kh(L;R) \to \Kh(L;R)$ satisfying $X_q\circ X_q = 0$, and the above isomorphism intertwines the basepoint operator on $\Kh(L;\Z/2)$ with the action of $X$ on $(\Z/2)[X]/X^2$. Both of these statements are false when the characteristic of $R$ is not $2$. The split union of a trefoil and an unknot provides a counterexample.

In this paper, we generalize these results to $\sl(P)$ link homology in characteristic $P$ when $P$ is prime. The (unreduced) $\sl(N)$ link homology of an oriented link $L$ with coefficients in a ring $R$ is a bigraded $R$-module $\KR_N(L;R)$, and a basepoint $q \in L$ determines a basepoint operator $X_q\colon \KR_N(L;R) \to \KR_N(L;R)$ satisfying $X_q^N = 0$. The reduced $\sl(N)$ link homology of $L$ equipped with $q$ is a bigraded $R$-module $\rKR_N(L,q;R)$, and an additional basepoint $r \in L$ determines a basepoint operator $X_r \colon \rKR_N(L,q;R) \to \rKR_N(L,q;R)$ satisfying $X_r^N = 0$. 

\begin{thm}\label{thm:mainThm}
	Let $P$ be a prime number, and let $R$ be a ring of characteristic $P$. \begin{enumerate}
		\item \emph{Reduced $\sl(P)$ link homology over $R$ is basepoint-independent}: for any two basepoints $q,r$ on an oriented link $L$, there is a bigraded $R$-module isomorphism \[
			\Phi\colon \rKR_P(L,q;R) \to \rKR_P(L,r;R)
		\]that satisfies $\Phi\circ X_r + X_q\circ \Phi = 0$.
		\item \emph{Unreduced $\sl(P)$ link homology over $R$ is determined by reduced $\sl(P)$ link homology}: for any basepoint $q$ on an oriented link $L$, there is a bigraded $R$-module isomorphism \[
			\Psi\colon \KR_P(L;R) \to \rKR_P(L,q;R) \otimes R[X]/X^P
		\]that satisfies $\Psi \circ X_q + X \circ \Psi = 0$.
	\end{enumerate}
\end{thm}

\begin{rem}
	The generalization of Theorem~\ref{thm:mainThm} to $\sl(N)$ link homology over a ring $R$ is false if the characteristic of $R$ does not divide $N$. A counterexample is provided by the split union of a trefoil and an unknot. The unreduced $\sl(N)$ link homology over $R$ of the right-handed trefoil, suppressing grading information, is given by the homology of the chain complex\[
		\begin{tikzcd}
			\displaystyle\frac{R[X]}{X^N} \ar[r,"NX^{N-1}"] & \displaystyle\frac{R[X]}{X^N} \ar[r] & 0 \ar[r] & \displaystyle \frac{R[X]}{X^N}
		\end{tikzcd}
	\]and its reduced $\sl(N)$ link homology is $R^3$. This computation can be done by hand. 

	The author does not know if the generalization of Theorem~\ref{thm:mainThm} to $\sl(N)$ link homology over $R$ holds if we only assume that the characteristic of $R$ divides $N$. 
\end{rem}

For $N \ge 2$, $\sl(N)$ link homology categorifies the $\sl(N)$ link polynomial $P_N \in \Z[q,q^{-1}]$ which is determined by the skein relation \[
	q^NP_N\left(\begin{gathered}
		\vspace{-3pt}
		\includegraphics[width=.03\textwidth]{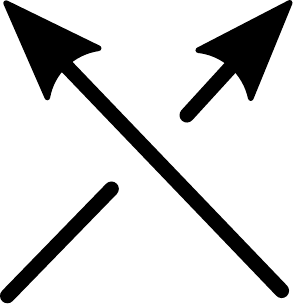}
	\end{gathered}\right) - q^{-N}P_N\left(\begin{gathered}
		\vspace{-3pt}
		\includegraphics[width=.03\textwidth]{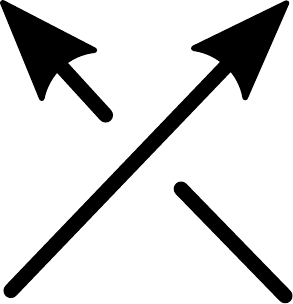}
	\end{gathered}\right) = (q + q^{-1})P_N\left(\begin{gathered}
		\vspace{-3pt}
		\includegraphics[width=.03\textwidth]{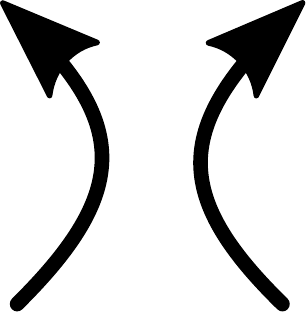}
	\end{gathered}\right)
\]and the normalization $P_N(\text{unknot}) = q^{N-1} + q^{N-3} + \cdots + q^{-(N-1)}$. Khovanov and Rozansky first defined $\sl(N)$ link homology using matrix factorizations over a field of characteristic zero in \cite{MR2391017}. There have since been many other approaches to constructing and generalizing $\sl(N)$ link homology (see the introduction of \cite{MR3787560} for a list and comparisons between them). We use the explicit combinatorial construction of $\sl(N)$ link homology over an arbitrary coefficient ring given by Robert and Wagner in \cite{MR4164001}. We provide the construction in full in section~\ref{sec:Preliminaries}. This version of $\sl(N)$ link homology agrees with an earlier version constructed by Queffelec and Rose \cite{MR3545951}, which agrees with Khovanov--Rozansky's original invariant when taken over a field of characteristic zero.

Let $N \ge 2$ be an integer and let $R$ be an arbitrary ring. Associated to a diagram $D$ of an oriented link $L$ is a chain complex $\KRC_N(D;R)$ over $R$ whose homology is $\KR_N(L;R)$. A basepoint $q$ on $D$ defines a chain map $X_q\colon \KRC_N(D;R) \to \KRC_N(D;R)$ satisfying $X_q^N = 0$ whose induced map on homology is the basepoint operator. The homology of the subcomplex $X_q^{N-1}\KRC_N(D;R)$ is the reduced $\sl(N)$ link homology of $L$ at $q$. 

If the characteristic of $R$ divides $N$, we define a chain map $\nabla\colon \KRC_N(D;R) \to \KRC_N(D;R)$ which behaves like a first-order differential operator. It satisfies a Leibniz product rule and is ultimately defined by the operator $\sum_i \frac{\partial}{\partial X_i}$ acting on a polynomial ring $R[X_1,\ldots,X_k]$. The construction of $\nabla$ is given in section~\ref{sec:ConstructionOfTheOp}. When $N = 2$, the operator agrees with Shumakovitch's acyclic differential $\nu$ on the Khovanov chain complex with mod $2$ coefficients, defined in \cite[section 3]{MR3205577}. In section~\ref{sec:ReducedslNLinkHomology}, we prove the following proposition (see Proposition~\ref{prop:inversesNprime}). 

\begin{prop}\label{prop:doubleCompositesNprime}
	Let $D$ be a diagram of an oriented link $L$ with a pair of basepoints $q,r \in D$, and let $\KRC_N(D;R)$ be its $\sl(N)$ chain complex over $R$. Assume that the characteristic of $R$ divides $N$ so that the operator $\nabla$ on $\KRC_N(D;R)$ is defined, and consider the chain maps \[
		\begin{tikzcd}
			X_q^{N-1}\KRC_N(D;R) \ar[r,swap,bend right = 30pt,"X_r^{N-1}\nabla^{N-1}"] & X_r^{N-1}\KRC_N(D;R) \ar[l,swap,bend right=30pt,"X_q^{N-1}\nabla^{N-1}"]
		\end{tikzcd}
	\]which map between the subcomplexes whose homologies are the reduced $\sl(N)$ link homologies of $L$ at the basepoints $q,r$. Then their composition is given by
	\[
		(X_q^{N-1}\nabla^{N-1})\circ(X_r^{N-1}\nabla^{N-1}) = \begin{cases}
			\Id & N \text{ is prime}\\
			0 & N \text{ is composite}.
		\end{cases}
	\]
\end{prop}

Primality of $N$ is involved because of the identity \[
	(N-1)!(N-1)! \bmod N = \begin{cases}
		1 & N \text{ is prime}\\
		0 & N \text{ is composite}
	\end{cases}
\]which follows from Wilson's theorem in elementary number theory. The first statement of Theorem~\ref{thm:mainThm} follows from Proposition~\ref{prop:doubleCompositesNprime}. The second statement follows from the first statement by a disjoint union formula for $\sl(N)$ link homology together with an argument appearing in the proof of \cite[Proposition 1.7]{MR3071132}. 

The operator $\nabla$ may also be defined on a generalization of $\sl(N)$ link homology called colored $\sl(N)$ link homology \cite{MR3234803,MR2863366}. The colored $\sl(N)$ link homology of an unknot equipped with an integer $k$ satisfying $1 \leq k \leq N-1$ coincides with the cohomology of the complex Grassmannian $\G(k,N)$. In section~\ref{sec:grassmannian}, we explicitly compute the operator $\nabla$ in this case in terms of the basis for $H^*(\G(k,N))$ given by Schur polynomials. Our formula for $\nabla$ on Schur polynomials is a special case of Proposition 1.1 of \cite{MR4098997}, in which the differential operator $\sum_i \frac{\partial}{\partial X_i}$ is used to prove a determinant conjecture of Stanley \cite{stanley2017schubert}. 

\vspace{20pt}

We point out two other instances where characteristic $2$ coefficients plays a special role in Khovanov homology. \begin{itemize}
	\item With $\Z/2$ coefficients, there is a spectral sequence \cite{MR2141852} from the reduced Khovanov homology of a link to the Heegaard Floer homology of the double branched cover of the mirror of the link. The spectral sequence appears not to generalize to other characteristics; instead, odd Khovanov homology, which agrees with Khovanov homology in characteristic $2$, was constructed in \cite{MR3071132} to replace Khovanov homology in a generalized spectral sequence. There have since been spectral sequences constructed from odd Khovanov homology to the monopole \cite{MR2764887} and instanton \cite{MR3394316} Floer homologies of the double branched cover. Scaduto's spectral sequence \cite{MR3394316} holds over any coefficient ring. 
	\item Lee defined a variant of Khovanov homology \cite{MR2173845} with coefficients in $\Q$. Rasmussen used Lee homology to define his $s$-invariant which he used in his celebrated combinatorial proof of the Milnor conjecture \cite{MR2729272}. Shumakovitch generalized Lee homology to characteristic $p$ coefficients, for $p$ an odd prime, and showed that the same basic features of Lee homology remain valid \cite{MR3205577}. Shumakovitch used this result and Lee's work on alternating links to show there is no $p$-torsion, for $p$ an odd prime, in the integral Khovanov homology of any alternating link, and in contrast, the integral Khovanov homology of any alternating link except for forests of unknots has $2$-torsion \cite[Corollaries 2 and 5]{MR3205577}.
\end{itemize}

\begin{rem}
	There is a spectral sequence from Khovanov homology to singular instanton homology \cite{MR2805599} defined over any coefficient ring. Kronheimer--Mrowka used this spectral sequence to prove that Khovanov homology detects the unknot. Reduced and unreduced singular instanton homology exhibit a behavior analogous to Khovanov homology. With coefficients in $\Z/2$, the dimension of unreduced singular instanton homology is twice the dimension of reduced singular instanton homology by \cite[Lemma 7.7]{MR3880205}. This statement is false over other characteristics. 
\end{rem}
\begin{rem}
	Knot Floer homology \cite{MR2065507,MR2704683} appears to behave uniformly over different characteristics, unlike Khovanov homology and singular instanton homology. It is an open problem to explain (or disprove) the apparent absence of torsion in the knot Floer homology of knots in $S^3$. 
\end{rem}

\theoremstyle{definition}
\newtheorem*{ack}{Acknowledgments}
\begin{ack}
	I thank Christian Gaetz, Lukas Lewark, and Roger Van Peski for helpful discussions and correspondences. I would also like to thank my advisor Peter Kronheimer for his continued guidance, support, and encouragement. This material is based upon work supported by the NSF GRFP through grant DGE-1745303.
\end{ack}

\section{Preliminaries}\label{sec:Preliminaries}

In this section, we review the construction of $\sl(N)$ link homology, following \cite{MR4164001}. For readers familiar with Khovanov homology but unfamiliar with $\sl(N)$ link homology, we briefly summarize the construction. Let $D$ be an oriented link diagram and let $R$ be a ring. If $c(D)$ denotes the set of crossings of $D$, then to each function $v\colon c(D) \to \{0,1\}$, we associate an \textit{$\sl(N)$ MOY graph} $D_v$ by applying a local change to each crossing of $D$ according to $v$. The object $D_v$ is analogous to a collection of circles in the plane obtained by a complete resolution in Khovanov homology. If $v,w\colon c(D)\to\{0,1\}$ agree at all crossings except at one where $v(c) = 0$ and $w(c) = 1$, then there is an \textit{$\sl(N)$ foam} $F_{vw}\colon D_v \to D_w$ which is a $2$-dimensional object analogous to the $2$-dimensional merge and split cobordisms in Khovanov homology. 

Associated to each $D_v$ is a free $R$-module $\sr F_N(D_v;R)$, analogous to the free $R$-module of rank $2^k$ associated to a collection of $k$ circles in the plane in Khovanov homology. Associated to each $F_{vw}\colon D_v\to D_w$ is an $R$-module map $\sr F_N(F_{vw};R) \colon \sr F_N(D_v;R) \to \sr F_N(D_w;R)$, analogous to the merge and split maps defined explicitly in Khovanov homology. The underlying $R$-module of the $\sl(N)$ chain complex $\KRC_N(D;R)$ is the direct sum of the $R$-modules $\sr F_N(D_v;R)$, and the differential is a signed sum of the maps $\sr F_N(F_{vw};R)$, just as in Khovanov homology. 

In section~\ref{subsec:FoamsAndMOYGraphs}, we review $\sl(N)$ foams and $\sl(N)$ MOY graphs. In section~\ref{subsec:TheUniversalConstruction}, we provide the construction of the $R$-module $\sr F_N(\Gamma;R)$ associated to an $\sl(N)$ MOY graph $\Gamma$, and the $R$-module map $\sr F_N(F;R)\colon \sr F_N(\Gamma;R) \to \sr F_N(\Gamma';R)$ associated to an $\sl(N)$ foam $F\colon \Gamma \to \Gamma'$. The construction is by a general procedure called the \textit{universal construction}, which requires as input a number for each \textit{closed} $\sl(N)$ foam. Robert--Wagner \cite{MR4164001} give a combinatorial formula for associating numbers to closed foams, which we explain in section~\ref{subsec:RobertWagnerEvaluation}. Some important basic properties of $\sr F_N(\Gamma;R)$ are reviewed in section~\ref{subsec:MOYCalculus}, and the construction of $\sl(N)$ link homology is given in section~\ref{subsec:slNLinkHomology}. In section~\ref{subsec:BasepointOperators}, we define basepoint operators and reduced $\sl(N)$ link homology. The version of reduced $\sl(N)$ link homology we use agrees with the more general constructions in \cite{MR3458146,MR3982970}.

\subsection{Foams and MOY graphs}\label{subsec:FoamsAndMOYGraphs}

\begin{df}
	A \textit{closed $\sl(N)$ foam $F$ (in $\R^3$)} consists of: \begin{itemize}[noitemsep]
		\item A compact space $F^2 \subset \R^3$ with compact subsets $F^0 \subset F^1 \subset F^2$ such that \begin{itemize}[noitemsep]
			\item $F^0$ is a finite set of points called \textit{singular points}, 
			\item $F^1\setminus F^0$ is a smoothly embedded $1$-manifold with finitely many components, which are called \textit{bindings},
			\item $F^2\setminus F^1$ is a smoothly embedded $2$-manifold with finitely many components, which are called \textit{facets}.
		\end{itemize}
		\item An orientation of each facet, and an orientation of each binding. 
		\item A label $\ell(f) \in \{0,\ldots,N\}$ for each facet $f$. 
	\end{itemize}We require that \begin{itemize}[noitemsep]
		\item Points on bindings and singular points have neighborhoods in $\R^3$ that intersect $F^2$ in the following local models: \[
			\labellist
			\pinlabel $\bullet$ at 60 87
			\pinlabel $\bullet$ at 287 57
			\endlabellist
			\begin{gathered}
				\includegraphics[width=.2\textwidth]{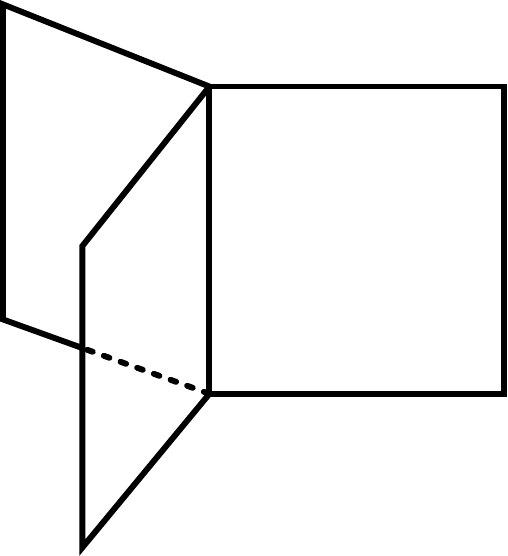}
			\end{gathered} \qquad\qquad \begin{gathered}
				\includegraphics[width=.17\textwidth]{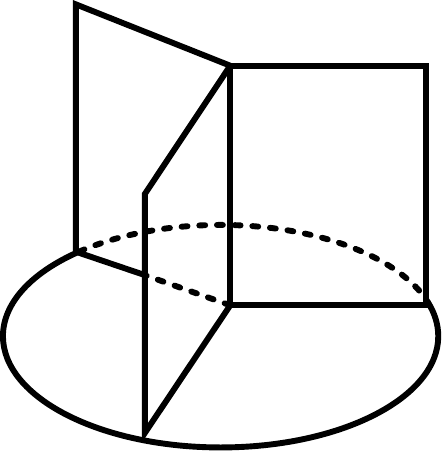}
			\end{gathered}
		\]In other words, every point on a binding has a neighborhood in which $F^2$ is an interval times a neighborhood of a degree $3$ vertex of a planar graph, and every singular point has a neighborhood in which $F^2$ is the cone on the $1$-skeleton of a standard tetrahedron.
		\item The orientation of each binding must agree with the orientations of exactly two of its three adjacent facets. The sum of the labels of those two facets must equal the label of the third. 
	\end{itemize}
\end{df}
\begin{rem}
	We will be primarily interested in the underlying stratified space $F^0 \subset F^1 \subset F^2$ of a closed foam together with the labels on its facets, rather than the specific way the foam is embedded in $\R^3$. The only relevant piece of information given by the embedding is the cyclic ordering of the three facets adjacent to each binding (which by convention is given by the left-hand rule). 
\end{rem}

\begin{df}
	An $\sl(N)$ \textit{MOY graph $\Gamma$ (in $\R^2$)} is an oriented trivalent graph embedded in the plane, where self-loops, multiple edges, and circles without vertices are permitted. Each vertex must have both incoming and outgoing edges, which is to say that sources and sinks are not permitted. 
	Each edge $e$ is given a label $\ell(e) \in \{0,\ldots,N\}$ such that at every vertex, the sum of the labels of the incoming edges equals the sum of the labels of the outgoing edges. 
\end{df}

A closed $\sl(N)$ foam $F$ in $\R^3$ is said to \textit{transversely} intersect a given plane in $\R^3$ if the plane is disjoint from all of the singular points of $F$ and transversely intersects each binding and facet. If $F$ intersects a given oriented plane $P$ transversely, then the intersection $\Gamma = F \cap P$ is an $\sl(N)$ MOY graph in $P$. The labels on the edges of $\Gamma$ are inherited from the labels on the facets of $F$, and orientations of edges of $\Gamma$ are induced from those of the facets of $F$ using the orientation of $P$. 

\begin{df}
	Let $F$ be a closed $\sl(N)$ foam in $\R^3$ that intersects both $\R^2 \x 0$ and $\R^2 \x 1$ transversely, and let $\Gamma_0 = F \cap (\R^2 \x 0)$ and $\Gamma_1 = F \cap (\R^2 \x 1)$ be the induced $\sl(N)$ MOY graphs where $\R^2 \x 0$ and $\R^2 \x 1$ are both oriented in the same standard direction. The intersection $F \cap (\R^2 \x [0,1])$ is said to be an \textit{$\sl(N)$ foam from $\Gamma_0$ to $\Gamma_1$}. An $\sl(N)$ foam from $\Gamma_0$ to $\Gamma_1$ also has facets, bindings, singular points, orientations of bindings and facets, and labels, which are defined in the same way as in the closed case $\Gamma_0 = \emp = \Gamma_1$. We use the notation $G\colon \Gamma_0 \to \Gamma_1$ to indicate that $G$ is an $\sl(N)$ foam from $\Gamma_0$ to $\Gamma_1$. Unless explicitly specified to be closed, an $\sl(N)$ foam refers to an $\sl(N)$ foam between $\sl(N)$ MOY graphs. 
\end{df}

\begin{df}
	Let $F$ be an $\sl(N)$ foam. A set of \textit{dots} on $F$ is a finite set of weighted points on the facets of $F$ where the weight $w(d)\in \Z$ of a dot $d$ lying on a facet $f$ must satisfy $1 \leq w(d) \leq \ell(f)$. We let $\bo{F}$ denote an $\sl(N)$ foam equipped with a set of dots, which we call a \textit{dotted} $\sl(N)$ foam. Two dotted foams with the same underlying foam with the same number of dots of each weight on each facet are thought of as the same. Equivalently, we may think of dots as free to move within their facet, as long as they remain disjoint. 
\end{df}

\begin{df}
	Let $R$ be a ring. Let $\RFoam_N$ be the category whose objects are $\sl(N)$ MOY graphs and whose set of morphisms $\RFoam_N(\Gamma_0,\Gamma_1)$ from $\Gamma_0$ to $\Gamma_1$ is the free $R$-module with basis the set of dotted $\sl(N)$ foams $\bo{F}$ from $\Gamma_0$ to $\Gamma_1$, thought of up to isotopy rel boundary. The composition of a dotted foam $\bo{F}\colon \Gamma_0 \to \Gamma_1$ with a dotted foam $\bo{G}\colon \Gamma_1 \to \Gamma_2$ is the dotted foam $\bo{F} \cup \bo{G}\colon \Gamma_0 \to \Gamma_2$ given by gluing along $\Gamma_1$. Composition in $\RFoam_N$ is given by $R$-linearly extending the composition of dotted foams. In particular, composition is $R$-bilinear, so we have $R$-linear composition maps \[
		\RFoam_N(\Gamma_0,\Gamma_1) \otimes \RFoam_N(\Gamma_1,\Gamma_2) \to \RFoam_N(\Gamma_0,\Gamma_2).
	\]
\end{df}

We explain how to view dots as what are called \textit{decorations} in \cite{MR4164001}. Robert--Wagner work with \textit{decorated} foams, which are slightly more general than dotted foams, but it is more convenient for us to work with dotted foams and to view decorated foams as formal linear combinations of dotted foams. Let $f$ be a facet of a dotted foam $\bo{F}$, and suppose it has $k_w$ dots of weight $w$ for $w = 1,\ldots, \ell(f)$. We associate to the facet $f$ the homogeneous symmetric polynomial $P_f$ in $\ell(f)$ variables given by \[
	P_f = \prod_{w=1}^{\ell(f)} (e_w)^{k_w}
\]where $e_w$ is the $w$th elementary symmetric polynomial in $\ell(f)$ variables. Variables associated to different facets are thought of as distinct. Given a dotted foam $\bo{F}$, adding a dot of weight $w$ to the facet $f$ corresponds to multiplying $P_f$ by $e_w$. We denote this new dotted foam by $e_w\bo{F}$. With this notation, any dotted foam may be written as $PF$ where $F$ is the underlying foam without dots and $P = \prod_f P_f$ where $P_f$ is a product of elementary symmetric polynomials in $\ell(f)$ variables. 

\begin{df}
	Let $F\colon\Gamma_0 \to \Gamma_1$ be an $\sl(N)$ foam. Let $R[F]$ be the subring of the polynomial ring \[
		R[(X_f)_1,\ldots,(X_f)_{\ell(f)}\:|\: f\text{ is a facet of }F]
	\]consisting of polynomials that are symmetric in the variables $(X_f)_1,\ldots,(X_f)_{\ell(f)}$ for each $f$. Alternatively, we may view $R[F]$ as the tensor product over all facets \[
		R[F] = \textstyle\bigotimes_f R[(X_f)_1,\ldots,(X_f)_{\ell(f)}]^{\fk S_{\ell(f)}}.
	\]By the fundamental theorem of elementary symmetric polynomials, any polynomial $Q_f \in R[(X_f)_1,\ldots,(X_f)_{\ell(f)}]^{\fk S_{\ell(f)}}$ is a unique $R$-linear combination of products of $e_1,\ldots,e_{\ell(f)}$. We define $Q_fF$ to be the element of $\RFoam_N(\Gamma_0,\Gamma_1)$ given by $R$-linearly extending the actions of $e_1,\ldots,e_{\ell(f)}$, which are given by the addition of dots of the appropriate weight to $f$. We then extend this action to define $QF \in \RFoam_N(\Gamma_0,\Gamma_1)$ for any $Q \in R[F]$.
\end{df}

Lastly, we define the degree of a dotted foam. 

\begin{df}
	Let $F\colon \Gamma_0 \to \Gamma_1$ be an $\sl(N)$ foam. We define the \textit{(quantum) degree} of $F$ to be the following alternating sum \[
		\deg(F) = - \sum_{\text{facets } f} d(f) + \sum_{\substack{\text{interval}\\\text{bindings }i}} d(i) - \sum_{\substack{\text{singular}\\\text{points }p}} d(p)
	\]where \begin{itemize}[noitemsep]
		\item if $f$ is a facet with $\ell(f) = a$, then $d(f) = a(N - a)\chi(f)$ where $\chi$ is the Euler characteristic,
		\item if $i$ is a binding diffeomorphic to an interval (as opposed to a circle) whose adjacent facets are labeled $a,b$, and $a + b$, then $d(i) = ab + (a + b)(N - a - b)$,
		\item if $p$ is a singular point whose adjacent facets are labeled $a,b,c,a + b,b + c,a + b + c$, then $d(p) = ab + bc + ac + (a + b + c)(N - a - b - c)$. 
	\end{itemize}If $\bo{F} = PF$ is a dotted foam where $P \in R[F]$ is homogeneous of ordinary degree $k$, then we define $\deg(\bo{F}) = \deg(F) + 2k$. 
\end{df}

The degree of dotted foams will ultimately define the quantum grading on $\sl(N)$ link homology. If $P$ is a homogeneous polynomial of ordinary degree $k$, we will refer to $2k$ as its \textit{quantum degree}. The degree of dotted foams induces a $\Z$-grading on $\RFoam_N(\Gamma_0,\Gamma_1)$. One may show that degree is additive under composition, so the composition maps in $\RFoam_N$ preserve $\Z$-gradings.

\subsection{Robert--Wagner evaluation}\label{subsec:RobertWagnerEvaluation}

Let $\bo{F}$ be a closed dotted $\sl(N)$ foam. Robert--Wagner \cite{MR4164001} define an \textit{evaluation} of $\bo{F}$\[
	\langle \bo{F}\rangle \in \Z[X_1,\ldots,X_N]^{\fk S_N}
\]which takes the form of a homogeneous symmetric polynomial of quantum degree $\deg(\bo{F})$ in the $N$ variables $X_1,\ldots,X_N$ with integral coefficients. The Robert--Wagner evaluation of closed dotted foams is explicit and is given in terms of colorings of $F$, which we now define. 

\begin{df}
	Let $\P = \{1,\ldots,N\}$ be a set of \textit{pigments}. A \textit{coloring} $c$ of a closed $\sl(N)$ foam $F$ consists of a choice of an $\ell(f)$-element subset $c(f)\subset\P$ for each facet $f$ of $F$, subject to the following constraints. If $f,g,h$ are the three facets adjacent to a given binding where $\ell(f) + \ell(g) = \ell(h)$, then we require that $c(f) \cup c(g) = c(h)$. Note that $c(f)$ and $c(g)$ must therefore be disjoint. 
\end{df}

We note that colorings are defined for foams, independent of dots. Given a closed dotted foam $\bo{F} = PF$ and a coloring $c$ of $F$, Robert--Wagner define a homogeneous rational function in the variables $X_1,\ldots,X_N$ with integral coefficients \[
	\langle \bo{F},c\rangle = (-1)^{s(F,c)} \frac{P(\bo{F},c)}{Q(F,c)}
\]where $P(\bo{F},c)$ is a polynomial, $Q(F,c)$ is a rational function, and $s(F,c)$ is an integer, all defined below. As indicated by our notation, $P(\bo{F},c)$ depends on both the dots and the coloring, whereas $Q(F,c)$ and the sign $s(F,c)$ depend only on the coloring and the underlying closed foam $F$. The Robert--Wagner evaluation is then defined to be \[
	\langle\bo{F}\rangle = \sum_c \langle \bo{F},c\rangle
\]where the sum is over all colorings of $F$. It is shown in \cite[Corollary 2.17]{MR4164001} that $\langle \bo{F}\rangle$ is symmetric and in \cite[Proposition 2.18]{MR4164001} that $\langle \bo{F}\rangle$ is an integral polynomial rather than just a rational function. Its quantum degree is computed to be $\deg(\bo{F})$ in \cite[Lemma 2.14]{MR4164001}. The evaluation is multiplicative under disjoint union. 

The polynomial $P(\bo{F},c)$ is defined to be \[
	P(\bo{F},c) = \prod_f P_f(X_{c(f)}) \in \Z[X_1,\ldots,X_N]
\]where $P_f(X_{c(f)})$ is obtained by replacing the formal variables $(X_f)_1,\ldots,(X_f)_{\ell(f)}$ of $P_f$ with the variables $X_i$ for $i \in c(f)$. Since $P_f$ is symmetric, the polynomial $P_f(X_{c(f)})$ is well-defined. 

The arguments in this paper will not use how the sign $s(F,c)$ is defined, and we will use the definition of $Q(F,c)$ only to see that it is a product of terms of the form $X_i - X_j$ and their inverses. We include their definitions here for the sake of completeness. Given a pigment $i \in \P$ and a coloring $c$ of a closed foam $F$, the \textit{monochrome surface} $F_i(c)$ is the closed orientable surface obtained by taking the union of the facets $f$ of $F$ for which $i \in c(f)$. Given two pigments $i,j\in \P$ where $i < j$, the \textit{bichrome surface} $F_{ij}(c)$ is the closed orientable surface obtained by taking the union of the facets $f$ of $F$ for which exactly one of $i$ and $j$ lies in $c(f)$. In this situation, consider a binding with adjacent facets $f,g,h$ such that $i \in c(f),j\in c(g)$, and $i,j\in c(h)$. The binding is said to be \textit{positive with respect to $i < j$} if the cyclic order (induced by the left-hand rule) of the facets around the binding is $f,g,h$ rather than $g,f,h$. The union of the bindings that are positive with respect to $i < j$ is a collection of simple closed curves in the bichrome surface $F_{ij}(c)$, and the number of such positive simple closed curves is defined to be $\theta_{ij}^+(c)$. With these definitions, we have \[
	Q(F,c) = \prod_{\substack{i,j\in \P\\i<j}} (X_i - X_j)^{\chi(F_{ij}(c))/2} \qquad\qquad s(F,c) = \sum_{i\in \P} i\cdot\frac{\chi(F_i(c))}{2} + \sum_{\substack{i,j\in\P\\i<j}} \theta_{ij}^+(F,c)
\]where $\chi$ denotes Euler characteristic. 

\subsection{The universal construction}\label{subsec:TheUniversalConstruction}

The universal construction \cite[Proposition 1.1]{MR1362791} is a general procedure of using an ``evaluation'' of ``closed'' objects to produce a functor on an associated ``cobordism'' category. In our setup, the closed objects are closed dotted $\sl(N)$ foams, our evaluation is given by the Robert--Wagner evaluation, and our cobordism category is $\RFoam_N$ where $R$ is a fixed ring. We apply the universal construction to obtain a functor $\sr F_N(-;R)\colon\RFoam_N\to\RMod$ where $\RMod$ is the category of $R$-modules. 

There is a unique map $\Z[X_1,\ldots,X_N] \to R$ that sends each indeterminant $X_i$ to $0$. Let $\langle \bo{F}\rangle_R$ denote the image of the Robert--Wagner evaluation $\langle \bo{F}\rangle \in \Z[X_1,\ldots,X_N]$ in $R$ where $\bo{F}$ is a closed dotted foam. For any $\sl(N)$ MOY graph $\Gamma$, the evaluation $\langle - \rangle_R$ defines an $R$-bilinear pairing $\langle -,-\rangle_R$ on $\RFoam_N(\emp,\Gamma)$ given by \[
	\langle \bo{F},\bo{G}\rangle_R = \langle \bo{F} \cup \widebar{\bo{G}}\rangle_R
\]where $\widebar{\bo{G}}\colon \Gamma \to \emp$ is obtained from $\bo{G}\colon \emp \to\Gamma$ by mirroring across a plane. Let $\sr F_N(\Gamma;R)$ be the quotient of $\RFoam_N(\emp,\Gamma)$ by the kernel of this pairing. Explicitly, an $R$-linear combination $\sum_i a_i \bo{F}_i$ represents zero in $\sr F_N(\Gamma;R)$ if and only if for every dotted foam $\bo{G}\colon \emp \to \Gamma$, we have \[
	\textstyle\sum_i a_i \langle \bo{F}_i, \bo{G}\rangle_R = \sum_i a_i \langle \bo{F}_i \cup \widebar{\bo{G}}\rangle_R = 0. 
\]The bilinear form $\langle-,-\rangle_R$ induces a nondegenerate pairing on $\sr F_N(\Gamma;R)$. Furthermore, a dotted foam $\bo{G}\colon \Gamma_0 \to \Gamma_1$ induces an $R$-linear map $\sr F_N(\bo{G};R)\colon \sr F_N(\Gamma_0;R) \to \sr F_N(\Gamma_1;R)$ given by $\bo{F}\mapsto \bo{F}\cup\bo{G}$. 

\begin{rem}\label{rem:mirroringGivesAdjoint}
	If $\bo{G}\colon \Gamma_0 \to \Gamma_1$ is a dotted foam and $\widebar{\bo{G}}\colon \Gamma_1 \to \Gamma_0$ is obtained by mirroring across a plane, then $\sr F_N(\bo{G};R)$ and $\sr F_N(\widebar{\bo{G}};R)$ are adjoint operators \[
		\langle \sr F_N(\bo{G};R)\bo{F},\bo{H}\rangle_R = \langle \bo{F} \cup \bo{G} \cup \widebar{\bo{H}}\rangle_R = \langle\bo{F},\sr F_N(\widebar{\bo{G}};R)\bo{H}\rangle_R.
	\]
\end{rem}

Note that for a closed dotted foam $\bo{F}$, the evaluation $\langle \bo{F}\rangle_R$ will vanish unless $\deg(\bo{F}) = 0$. The $\Z$-grading on $\RFoam_N(\emp,\Gamma)$ therefore descends to a $\Z$-grading on $\sr F_N(\Gamma;R)$ for which $\sr F_N(\bo{G};R)\colon \sr F_N(\Gamma_0;R) \to \sr F_N(\Gamma_1;R)$ is homogeneous of degree $\deg(\bo{G})$ for any dotted foam $\bo{G}\colon \Gamma_0\to \Gamma_1$.

\subsection{MOY calculus}\label{subsec:MOYCalculus}

Murakami, Ohtsuki, and Yamada \cite{MR1659228} associate to each $\sl(N)$ MOY graph $\Gamma$ an integral polynomial $\langle \Gamma\rangle \in \Z[q,q^{-1}]$ with nonnegative coefficients, which is an invariant of $\Gamma$ up to ambient isotopy in $\R^2$ preserving labels. 
If $D$ is a diagram for an oriented link $L$, then the $\sl(N)$ link polynomial of $L$ may be computed from the MOY evaluations of MOY graphs obtained from $D$ by the rules in Figure~\ref{fig:slNpolynomialFromMOY}.
The MOY evaluation is multiplicative under disjoint union and is uniquely determined by a number of local relations which are collectively referred to as \textit{MOY calculus} (see \cite[section 4.2]{MR3234803} for the list of relations and a proof of uniqueness). 

\begin{figure}[!ht]
	\centering \begin{align*}
		\begin{gathered}
		\includegraphics[width=.08\textwidth]{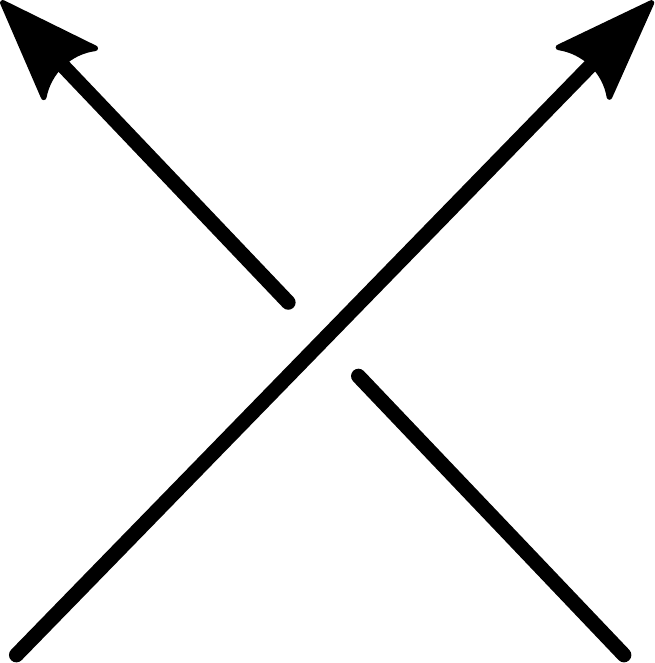}
		\end{gathered} &= q^{N-1}\begin{gathered}
		\includegraphics[width=.08\textwidth]{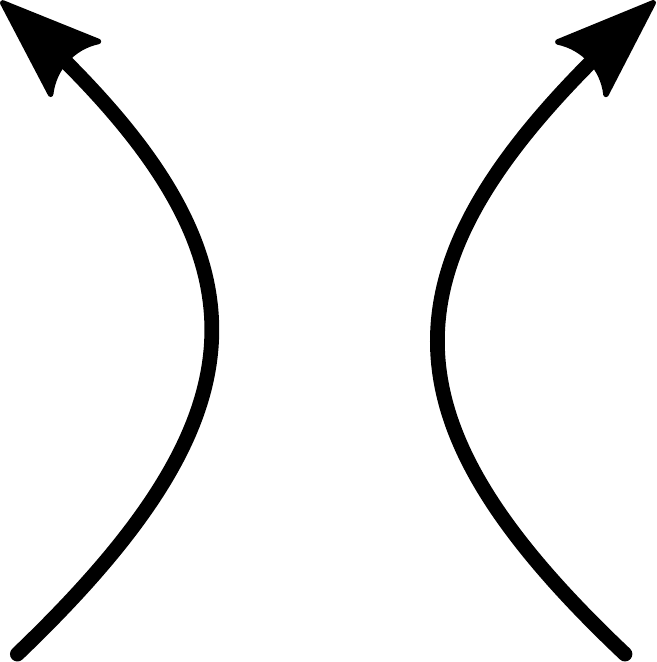}
		\end{gathered} - q^{+N} \begin{gathered}
			\labellist
			\pinlabel $2$ at 130 90
			\endlabellist
			\includegraphics[width=.08\textwidth]{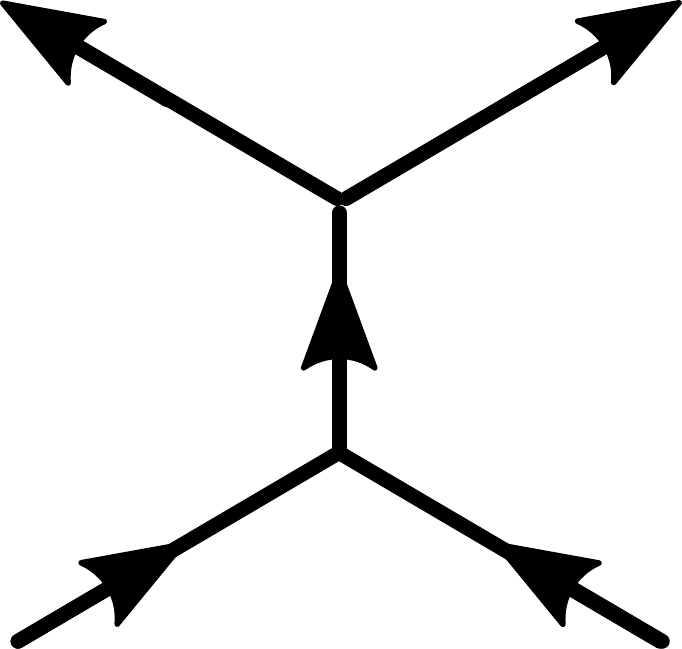}
		\end{gathered}\\
\begin{gathered}
		\includegraphics[width=.08\textwidth]{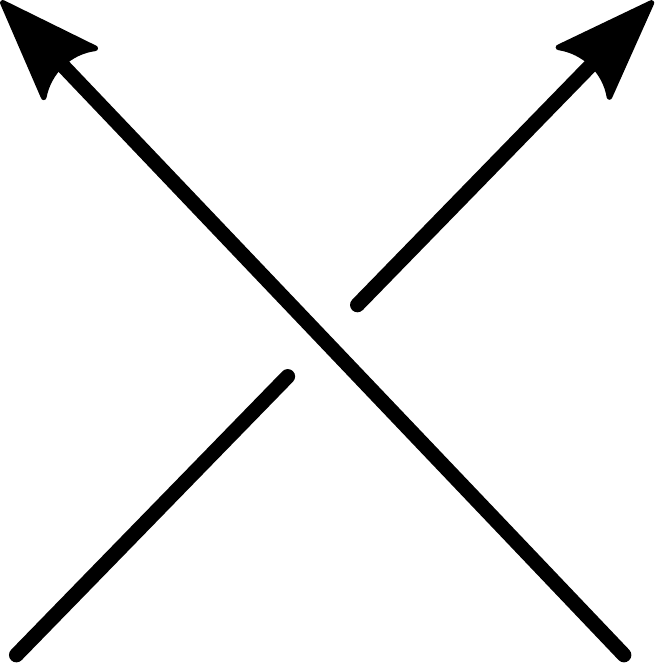}
	\end{gathered} &= q^{1 - N}\begin{gathered}
		\includegraphics[width=.08\textwidth]{orientedResolution.pdf}
	\end{gathered} - q^{-N} \begin{gathered}
		\labellist
			\pinlabel $2$ at 130 90
			\endlabellist
		\includegraphics[width=.08\textwidth]{otherResolution}
	\end{gathered}
	\end{align*}
	\vspace{-10pt}
	\captionsetup{width=.8\linewidth}
	\caption{The $\sl(N)$ link polynomial from the MOY evaluation. An edge of an MOY graph without an explicit label is labeled $1$. }
	\label{fig:slNpolynomialFromMOY}
\end{figure}

For the reader's convenience, we provide a subset of those relations which uniquely determine the MOY evaluation for $\sl(N)$ MOY graphs with edges labeled $1$ or $2$ (see for example \cite[Figure 3]{MR2491657}). For each integer $k \ge 0$, let the \textit{quantum integer} $[k] \in \Z[q,q^{-1}]$ be \[
	[k] = \frac{q^k - q^{-k}}{q - q^{-1}} = q^{-(k-1)} + q^{-(k-3)} + \cdots + q^{k-3} + q^{k-1}
\]and for each integer $n \ge 0$, let the \textit{quantum binomial coefficient} $\qbinom{n}{k} \in \Z[q,q^{-1}]$ be \[
	\qbinom{n}{k} = \begin{cases}
		\displaystyle\frac{[n]!}{[k]![n-k]!} & 0 \leq k \leq n\\
		0 & \text{else}
	\end{cases}
\]where $[k]! = [k]\cdots[2][1]$. 

\theoremstyle{plain}
\newtheorem*{MOYCalc}{MOY Calculus \normalfont{(for graphs labeled by 1 and 2)}}

\begin{MOYCalc}
	The MOY evaluation $\langle \Gamma\rangle\in\Z[q,q^{-1}]$ of an $\sl(N)$ MOY graph whose edges are labeled $1$ or $2$ is determined by the local relations in Figure~\ref{fig:MOYcalculus}. 
\end{MOYCalc}

\begin{figure}[!ht]
	\centering
	\[
		\Big\langle\,\begin{gathered}
			\labellist
			\pinlabel $1$ at 120 100
			\endlabellist
			\includegraphics[width=.04\textwidth]{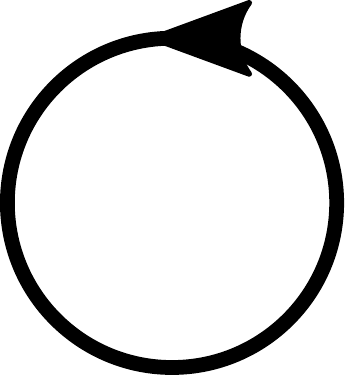}
		\end{gathered}\:\:\:\:\Big\rangle = \Big\langle\,\begin{gathered}
			\labellist
			\pinlabel $1$ at 120 100
			\endlabellist
			\includegraphics[width=.04\textwidth]{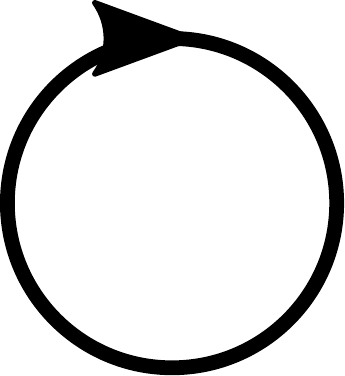}
		\end{gathered}\:\:\:\:\Big\rangle = [N] \:\Big\langle \qquad\Big\rangle \qquad\qquad \Big\langle\,\begin{gathered}
			\labellist
			\pinlabel $2$ at 120 100
			\endlabellist
			\includegraphics[width=.04\textwidth]{counterclockwiseCircle}
		\end{gathered}\:\:\:\:\Big\rangle = \Big\langle\,\begin{gathered}
			\labellist
			\pinlabel $2$ at 120 100
			\endlabellist
			\includegraphics[width=.04\textwidth]{clockwiseCircle}
		\end{gathered}\:\:\:\Big\rangle = \qbinom{N}{2}\:\Big\langle \qquad\Big\rangle
	\]\[
		\Bigg\langle \:\begin{gathered}
			\labellist
			\pinlabel $2$ at 90 20
			\pinlabel $2$ at 90 220
			\endlabellist
			\includegraphics[width=.045\textwidth]{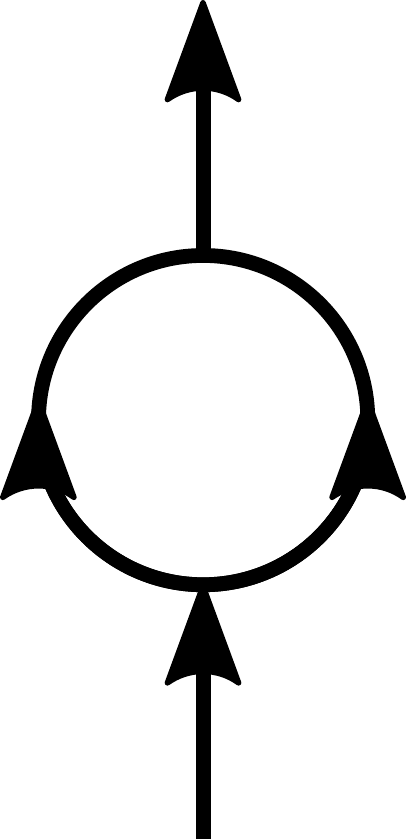}
		\end{gathered} \:\:\:\Bigg\rangle = [2]\: \Bigg\langle\:\: \begin{gathered}
			\labellist
			\pinlabel $2$ at 40 20
			\endlabellist
			\includegraphics[width=.0085\textwidth]{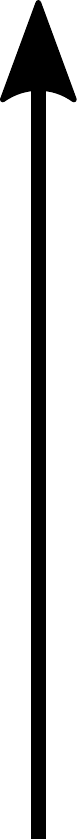}
		\end{gathered} \:\:\:\Bigg\rangle \qquad \qquad\qquad \Bigg\langle \:\begin{gathered}
			\labellist
			\pinlabel $2$ at 145 120
			\endlabellist
			\includegraphics[width=.045\textwidth]{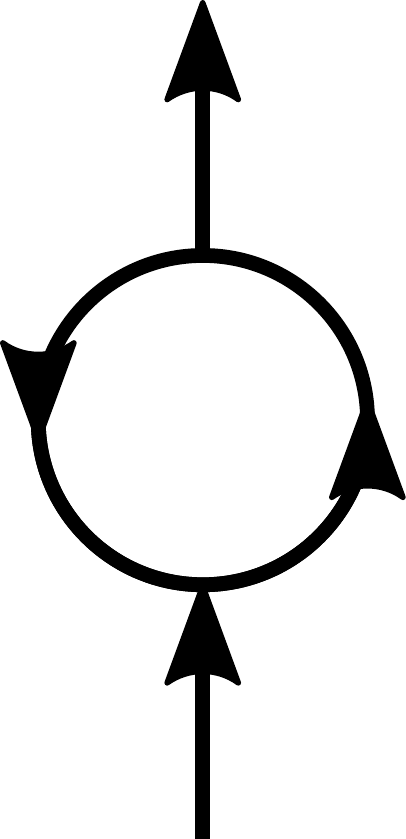}
		\end{gathered} \:\:\:\Bigg\rangle = [N-1]\: \Bigg\langle\:\: \begin{gathered}
			\includegraphics[width=.0085\textwidth]{verticalMOY}
		\end{gathered} \:\:\:\Bigg\rangle = \Bigg\langle \:\:\:\begin{gathered}
			\labellist
			\pinlabel $2$ at -25 120
			\endlabellist
			\includegraphics[width=.045\textwidth]{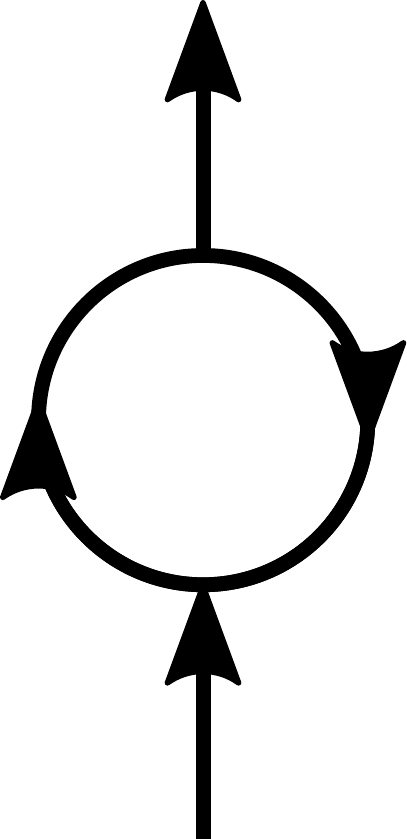}
		\end{gathered} \:\Bigg\rangle
	\]\[
		\Bigg\langle\: \begin{gathered}
			\labellist
			\pinlabel $2$ at 10 100
			\pinlabel $2$ at 190 100
			\endlabellist
			\includegraphics[width=.08\textwidth]{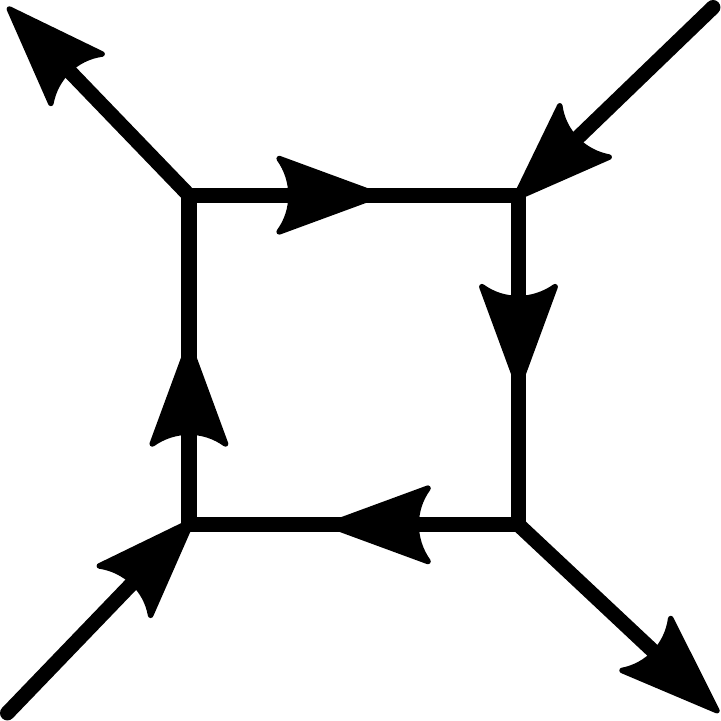}
		\end{gathered} \:\:\Bigg\rangle = \Bigg\langle\: \begin{gathered}
			\includegraphics[width=.08\textwidth]{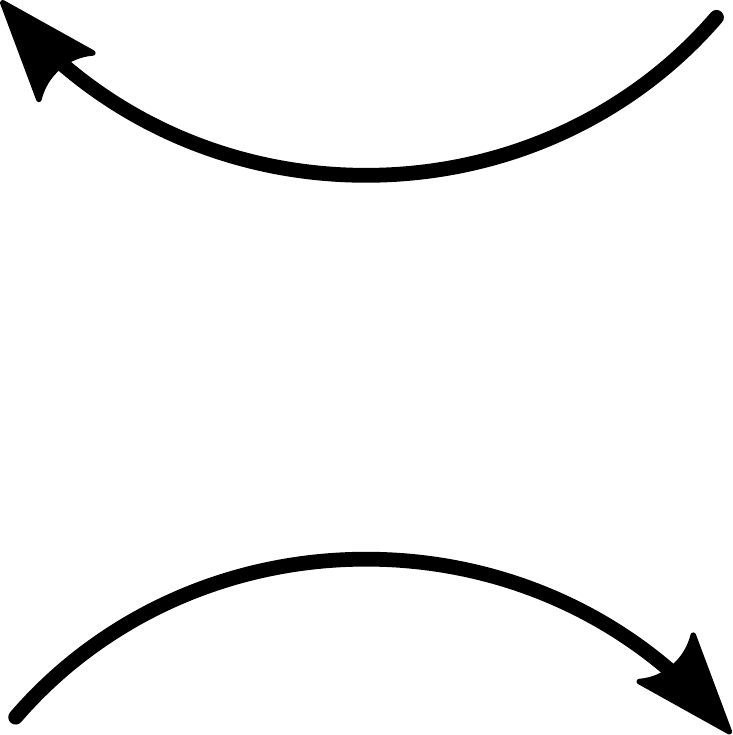}
		\end{gathered} \:\:\Bigg\rangle + [N-2] \:\Bigg\langle \: \begin{gathered}
			\includegraphics[width=.08\textwidth]{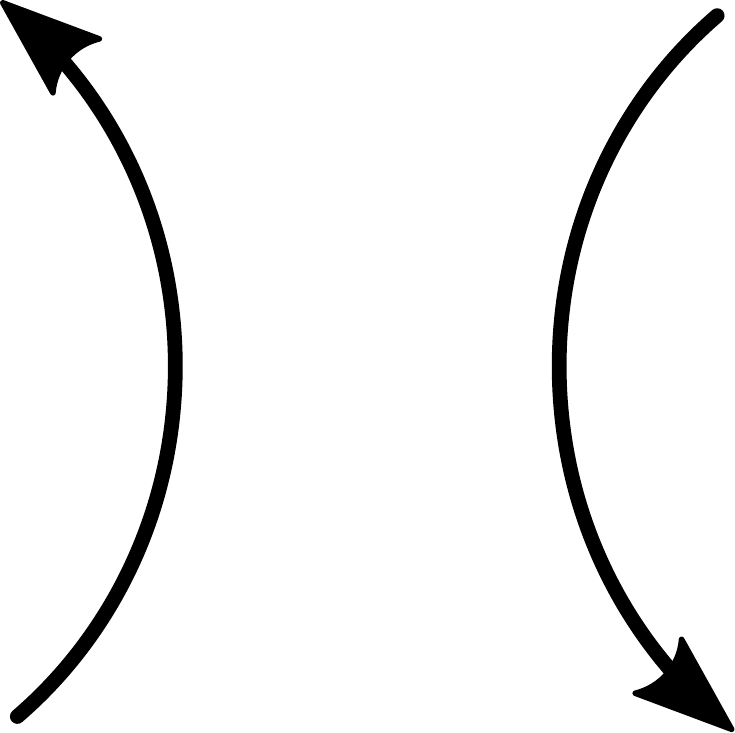}
		\end{gathered}\:\:\Bigg\rangle
	\]\[
		\Bigg\langle\:\:\begin{gathered}
			\labellist
			\pinlabel $2$ at 110 140
			\pinlabel $2$ at 110 260
			\pinlabel $2$ at 300 190
			\endlabellist
			\includegraphics[width=.12\textwidth]{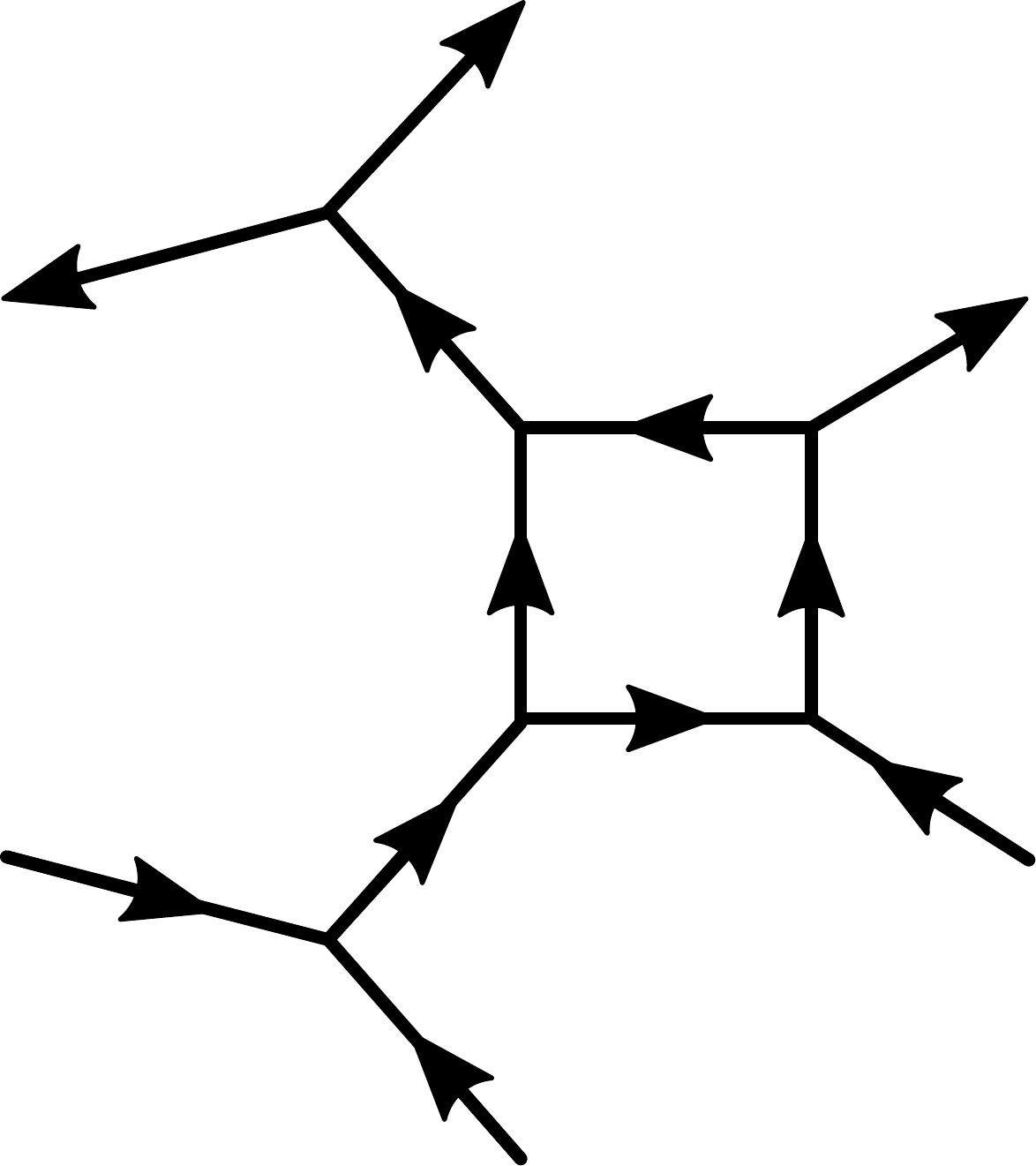}
		\end{gathered}\:\:\:\Bigg\rangle + \Bigg\langle\:\:\begin{gathered}
			\labellist
			\pinlabel $2$ at 310 190
			\endlabellist
			\includegraphics[width=.12\textwidth]{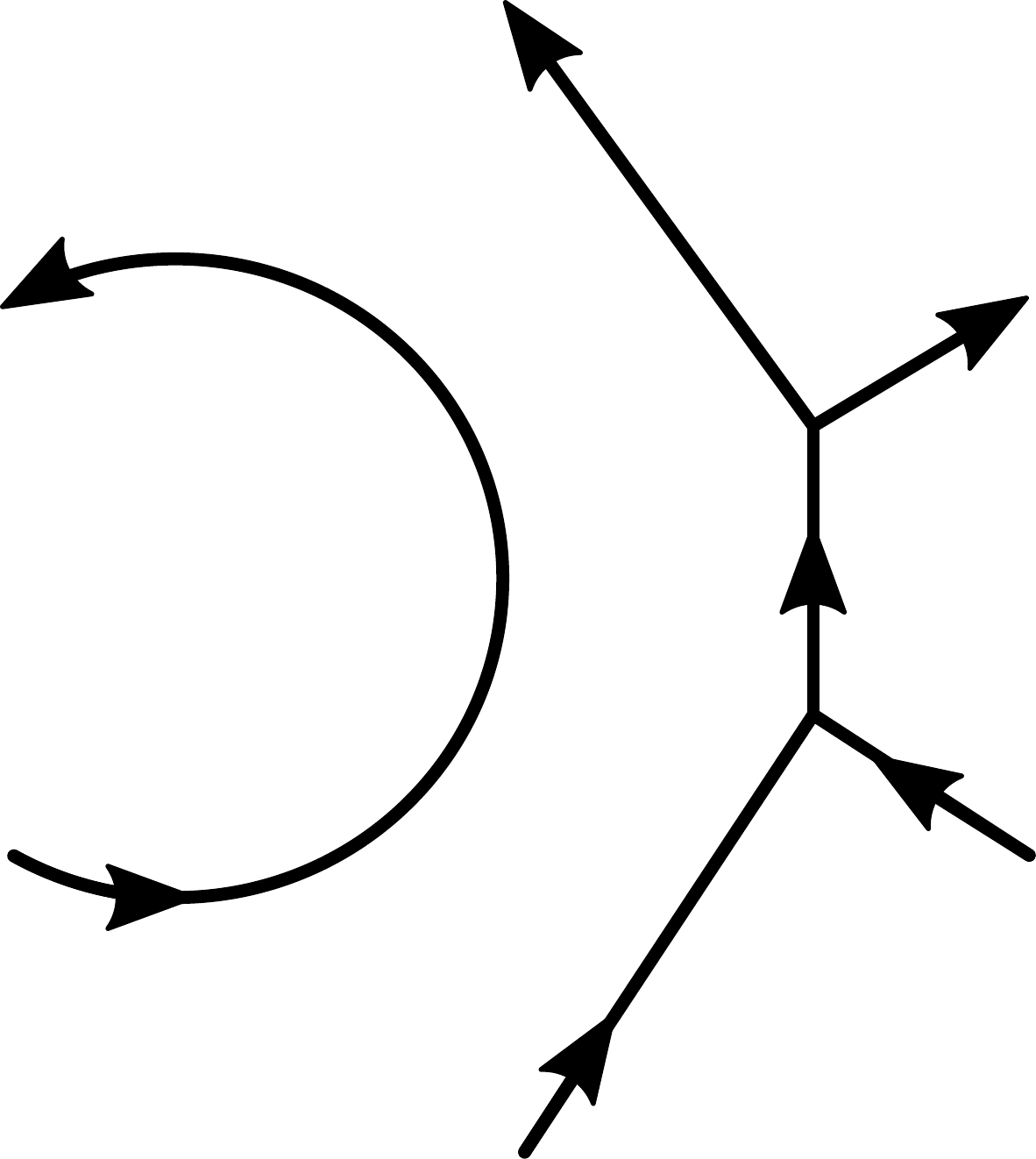}
		\end{gathered}\:\:\:\Bigg\rangle = \Bigg\langle\:\: \begin{gathered}
			\labellist
			\pinlabel $2$ at 35 190
			\pinlabel $2$ at 232 140
			\pinlabel $2$ at 232 260
			\endlabellist
			\includegraphics[width=.12\textwidth]{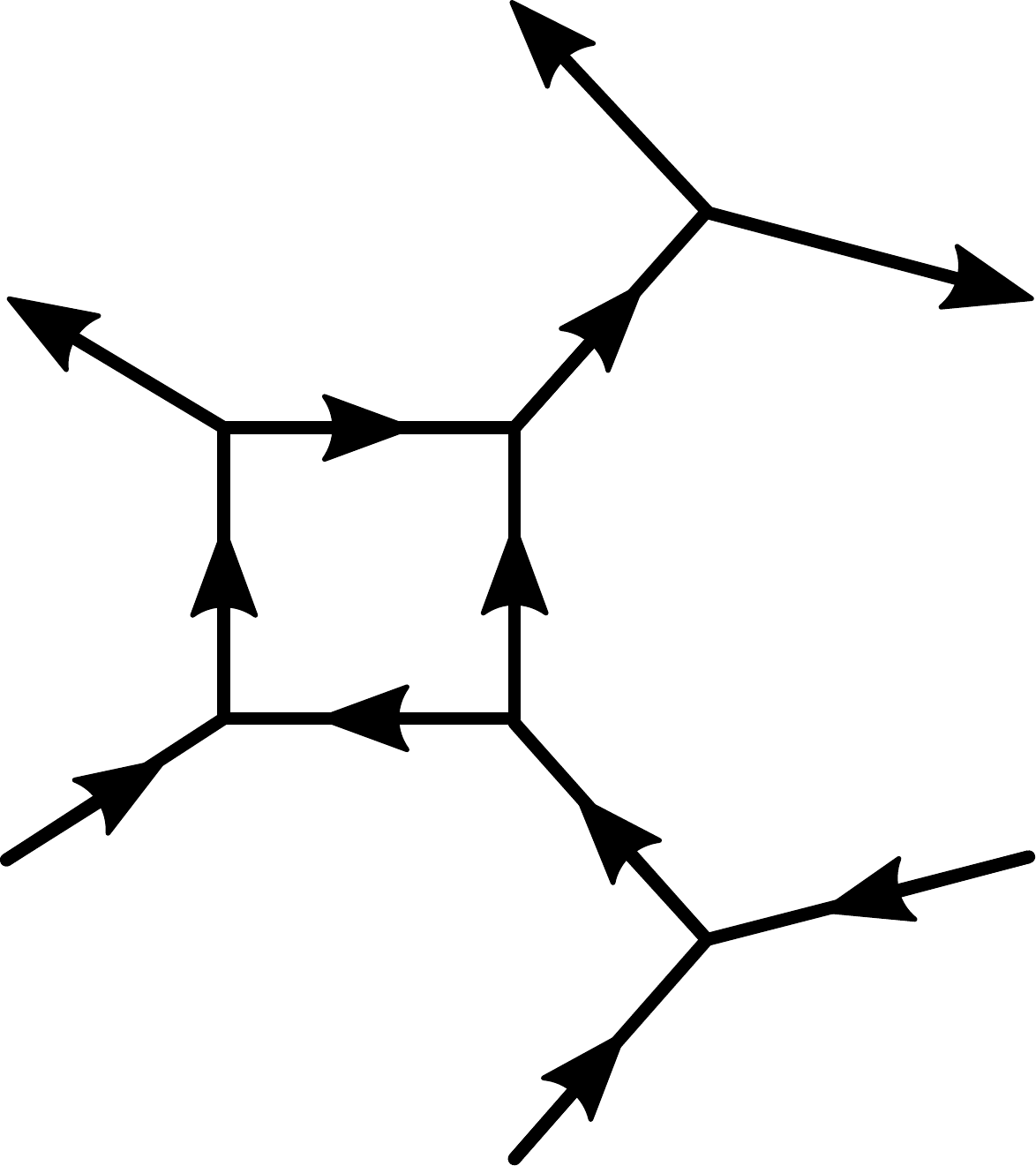}
		\end{gathered}\:\:\:\Bigg\rangle + \Bigg\langle\:\:\begin{gathered}
			\labellist
			\pinlabel $2$ at 35 190
			\endlabellist
			\includegraphics[width=.12\textwidth]{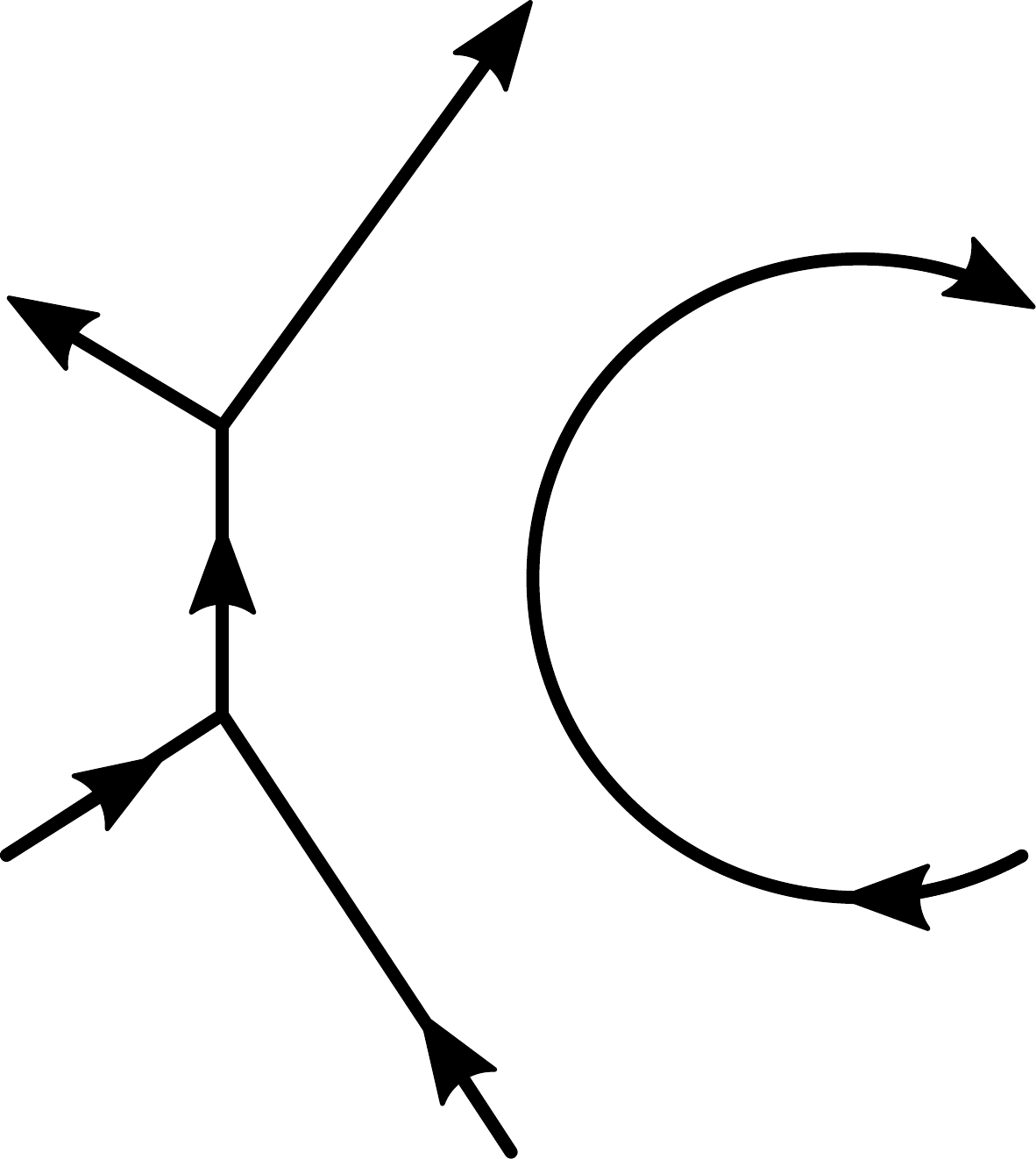}
		\end{gathered}\:\:\:\Bigg\rangle
	\]
	\vspace{-10pt}
	\captionsetup{width=.8\linewidth}
	\caption{Local relations in MOY calculus. An edge of an MOY graph without an explicit label is labeled $1$.}
	\label{fig:MOYcalculus}
\end{figure}

The following theorem is a consequence of the main result of \cite{MR4164001}. The \textit{$q$-graded rank} of a $\Z$-graded finitely generated free $R$-module $V = \bigoplus_i V_i$ is $\sum_i \rk(V_i)q^i \in \Z[q,q^{-1}]$. 

\begin{thm}[{\cite[Theorem 3.30]{MR4164001}}]
	For any $\sl(N)$ MOY graph $\Gamma$, the $\Z$-graded abelian group $\sr F_N(\Gamma;\Z)$ is finitely generated and free. Its $q$-graded rank is $\langle \Gamma \rangle \in \Z[q,q^{-1}]$. 
\end{thm}

Robert--Wagner's proof involves categorifying MOY calculus and implies the following corollary (see \cite[Remark 3.31]{MR4164001}).

\begin{cor}\label{cor:perfectPairing}
	The nondegenerate pairing $\langle-,-\rangle_\Z$ on $\sr F_N(\Gamma;\Z)$ is a perfect pairing, which is to say that the injective map \[
		\sr F_N(\Gamma;\Z) \to \Hom_\Z(\sr F_N(\Gamma;\Z),\Z) \qquad \bo{F} \mapsto (\bo{G}\mapsto \langle \bo{F},\bo{G}\rangle_\Z)
	\]is an isomorphism. 
\end{cor}
\begin{proof}
	Suppose $V,W$ are abelian groups equipped with nondegenerate pairings, and suppose $f\colon V \to W$ and $g\colon W \to V$ are inverse isomorphisms. If $f$ and $g$ admit adjoints $f^\dagger$ and $g^\dagger$ with respect to the pairings, then the pairing on $V$ is perfect if and only if the pairing on $W$ is perfect. Indeed, it is straightforward to verify that the diagram \[
		\begin{tikzcd}
			V \ar[r] \ar[d,"g^\dagger"] & \Hom_\Z(V,\Z) \ar[d,"-\circ g"]\\
			W \ar[r] & \Hom_\Z(W,\Z)
		\end{tikzcd}
	\]commutes. Surjectivity of $V \to \Hom_\Z(V,\Z)$ then implies surjectivity of $W \to \Hom_\Z(W,\Z)$. 

	If $V$ is a finite direct sum $V = \bigoplus_\alpha V_\alpha$ and the pairing on $V$ is induced from nondegenerate pairings on the groups $V_\alpha$, then the pairing on $V$ is perfect if and only if the pairing on each $V_\alpha$ is perfect. Suppose $f\colon V \to W$ and $g\colon W \to V$ are given by maps $f_\alpha\colon V_\alpha \to W$ and $g_\alpha\colon W \to V_\alpha$. Then $f$ and $g$ admit adjoints if each $f_\alpha$ and $g_\alpha$ admits an adjoint. 

	For each $\sl(N)$ MOY graph $\Gamma$, Robert and Wagner construct maps $f_\alpha\colon V_\alpha \to \sr F_N(\Gamma;\Z)$ and $g_\alpha\colon \sr F_N(\Gamma;\Z) \to V_\alpha$ such that \begin{itemize}[noitemsep]
		\item each $V_\alpha$ is just $\sr F_N(\emp;\Z)$ with a shift in $\Z$-grading,
		\item the induced maps $\bigoplus_\alpha V_\alpha \to \sr F_N(\Gamma;\Z)$ and $\sr F_N(\Gamma;\Z) \to \bigoplus_\alpha V_\alpha$ are inverse isomorphisms, 
		\item each $f_\alpha$ and $g_\alpha$ is induced by a $\Z$-linear combination of dotted foams. 
	\end{itemize}Maps induced by dotted foams have adjoints (Remark~\ref{rem:mirroringGivesAdjoint}), so each $f_\alpha$ and $g_\alpha$ also admits an adjoint. The pairing on each $V_\alpha$ is perfect by explicitly computing that $\sr F_N(\emp;\Z) \cong \Z$ where the pairing is given by multiplication in $\Z$. It follows that the pairing on $\sr F_N(\Gamma;\Z)$ is perfect. 
\end{proof}

\begin{cor}\label{cor:FNwithRCoefficients}
	The $R$-linear map $\nu_\Gamma\colon\sr F_N(\Gamma;\Z) \otimes R \to \sr F_N(\Gamma;R)$ defined by $\bo{F}\otimes 1 \mapsto \bo{F}$ is a natural isomorphism. Hence, $\sr F_N(\Gamma;R)$ is finitely generated and free with $q$-graded rank $\langle \Gamma\rangle \in \Z[q,q^{-1}]$. The pairing on $\sr F_N(\Gamma;R)$ is also perfect. 
\end{cor}
\begin{proof}
	Since $\ZFoam_N(\emp,\Gamma) \otimes R \to \RFoam_N(\emp,\Gamma)$ given by $\bo{F}\otimes1\mapsto\bo{F}$ is an isomorphism, the commutative diagram \[
		\begin{tikzcd}
			\ZFoam_N(\emp,\Gamma) \otimes R \ar[r] \ar[d,two heads] & \RFoam_N(\emp,\Gamma) \ar[d,two heads]\\
			\sr F_N(\Gamma;\Z) \otimes R \ar[r,"\nu_\Gamma"] & \sr F_N(\Gamma;R)
		\end{tikzcd}
	\]implies that $\nu_\Gamma$ is surjective. 
	Let $r = \rk(\sr F_N(\Gamma;\Z))$. 
	By Corollary~\ref{cor:perfectPairing}, for every basis $f_1,\ldots,f_r$ of $\sr F_N(\Gamma;\Z)$, there is a dual basis $g_1,\ldots,g_r \in \sr F_N(\Gamma;\Z)$, which is to say that $\langle f_i,g_j \rangle_\Z = \delta_{ij}$ where $\delta_{ij}$ is the Kronecker delta. It follows that $\langle \nu_\Gamma(f_i \otimes 1),\nu_\Gamma(g_j\otimes 1) \rangle_R = \delta_{ij}$ as well, which implies that $\nu_\Gamma$ is injective. 

	The proof of Corollary~\ref{cor:perfectPairing} implies that the pairing $\langle -,-\rangle_R$ is perfect. Alternatively, the diagram \[
		\begin{tikzcd}
			\sr F_N(\Gamma;\Z) \otimes R \ar[r] \ar[d,"\nu_\Gamma"] & \Hom_\Z(\sr F_N(\Gamma;\Z),\Z) \otimes R \ar[d]\\
			\sr F_N(\Gamma;R) \ar[r] & \Hom_R(\sr F_N(\Gamma;R),R)
		\end{tikzcd}
	\]commutes. Surjectivity of the rightmost vertical map combined with Corollary~\ref{cor:perfectPairing} implies that $\langle-,-\rangle_R$ is perfect. 
\end{proof}

Suppose $\Gamma$ is an $\sl(N)$ MOY graph and $D \subset \R^2$ is a disc for which $\Gamma$ is disjoint from $\partial D$. Then $\Gamma_1 = \Gamma \cap D$ may be viewed as an $\sl(N)$ MOY graph. If $\Gamma_0 = \Gamma\setminus \Gamma_1$, then \[
	\langle\Gamma\rangle = \langle \Gamma_0\rangle\cdot\langle\Gamma_1\rangle \in \Z[q,q^{-1}].
\]The following corollary categorifies this equality. We first define a map \[
	\iota\colon\RFoam_N(\emp,\Gamma_1) \to \RFoam_N(\Gamma_0,\Gamma)
\]in the following way. Fix a dotted foam $\bo{G}\colon \emp \to \Gamma_1$, and then isotope $\bo{G} \subset \R^2 \x [0,1]$ rel boundary so that it is contained in $D \x [0,1]$. Now $\bo{G}$ is disjoint from the product cobordism $\Gamma_0 \x [0,1]$, so we set $\iota(\bo{G})$ to be their disjoint union viewed as a dotted cobordism $\Gamma_0 \to \Gamma$. 

\begin{cor}\label{cor:splitUnionFormulaForFN}
	There is an isomorphism $\sr F_N(\Gamma_0;R) \otimes \sr F_N(\Gamma_1;R) \to \sr F_N(\Gamma;R)$ given by $\bo{F}\otimes\bo{G} \mapsto \bo{F}\cup \iota(\bo{G})$.
\end{cor}
\begin{proof}
	Let \[
		\Phi\colon \RFoam_N(\emp,\Gamma_0)\otimes \RFoam_N(\emp,\Gamma_1) \to \RFoam_N(\emp,\Gamma)
	\]be the map given by $\bo{F}\otimes\bo{G}\mapsto \bo{F} \cup \iota(\bo{G})$. It is not clear \textit{a priori} that $\Phi$ descends to a map on quotients. To verify that a given element of $\RFoam_N(\emp,\Gamma)$ represents the trivial element of $\sr F_N(\Gamma;R)$, we must check that it pairs trivially with every element of $\RFoam_N(\emp,\Gamma)$, not just those in the image of $\Phi$.
	We use the fact that the pairings are perfect (Corollary~\ref{cor:perfectPairing}) to show that every class in $\sr F_N(\Gamma;R)$ is represented by an element in the image of $\Phi$. 

	Choose elements $a_1,\ldots,a_r \in \RFoam_N(\emp,\Gamma_0)$ and $f_1,\ldots,f_s\in \RFoam_N(\emp,\Gamma_1)$ which represent bases for $\sr F_N(\Gamma_0;R)$ and $\sr F_N(\Gamma_1;R)$. Then choose elements $b_1,\ldots,b_r \in \RFoam_N(\emp,\Gamma_0)$ and $g_1,\ldots,g_s \in \RFoam_N(\emp,\Gamma_1)$ that represent the dual bases. Then \[
		\langle \Phi(a_i \otimes f_k), \Phi(b_k \otimes g_\ell)\rangle_R = \delta_{ij}\delta_{k\ell}
	\]by multiplicativity of $\langle - \rangle_R$ under disjoint union. Using the fact that \[
		\rk(\sr F_N(\Gamma;R)) = \rk(\sr F_N(\Gamma_0;R))\cdot\rk(\sr F_N(\Gamma_1;R)),
	\]it follows that the elements $\Phi(a_i \otimes f_k)$ represent a basis for $\sr F_N(\Gamma;R)$. Thus, to prove that a given element of $\RFoam_N(\emp,\Gamma)$ represents the trivial class in $\sr F_N(\Gamma;R)$, it suffices to check that its pairing with each $\Phi(a_i \otimes f_k)$ vanishes. It now follows that $\Phi$ descends to a map on quotients. The map on quotients is an isomorphism since it sends a basis to a basis. 
\end{proof}

\subsection{$\sl(N)$ link homology}\label{subsec:slNLinkHomology}

Let $D$ be a diagram of an oriented link $L$ with crossings numbered $1,\ldots,n$, and let $R$ be a ring. The \textit{(Khovanov--Rozansky) $\sl(N)$ chain complex of $D$ with coefficients in $R$}, denoted $\KRC_N(D;R)$, is a finitely generated free $R$-module equipped with two $\Z$-gradings called its quantum and homological gradings. Its differential preserves quantum grading and increases homological grading by $1$. The definition given here purposefully mimics Bar-Natan's description of Khovanov homology \cite{MR1917056}, except that we follow the usual conventions for $\sl(N)$ link homology (see for example \cite{MR2391017}) which occasionally disagree with the usual conventions of Khovanov homology. 

Define a \textit{complete resolution} $v$ of $D$ to be a map $v\colon \{1,\ldots,n\} \to \{0,1\}$ which associates to each crossing either $0$ or $1$. Using the rule in Figure~\ref{fig:resolutions}, we associate to each complete resolution $v$ an $\sl(N)$ MOY graph $D_v$, where each edge of $D_v$ is labeled either $1$ or $2$. We associate to $D_v$ the finitely generated free $R$-module $\sr F_N(D_v;R)$ equipped with its $\Z$-grading which we call its \textit{quantum} grading or $q$-grading. We introduce a new $\Z$-grading on $\sr F_N(D_v;R)$ called its \textit{homological grading} or $h$-grading by declaring that $\sr F_N(D_v;R)$ lies entirely in homological grading $0$. For any integers $i,j$, let $h^iq^j\sr F_N(D_v;R)$ be the bigraded $R$-module whose bigraded summand in $(h,q)$-grading $(x + i,y + j)$ is equal to the bigraded summand of $\sr F_N(D_v;R)$ in $(h,q)$-grading $(x,y)$.
If $D$ has $n_+$ positive crossings and $n_- = n - n_+$ negative crossings, then define \[
	\KRC_N(D;R) = h^{-n_+}q^{Nn_+ - (N-1)n_-} \bigoplus_v h^{|v|}q^{-|v|} \sr F_N(D_v;R)
\]where $|v| = \sum_{c=1}^n v(c)$. 

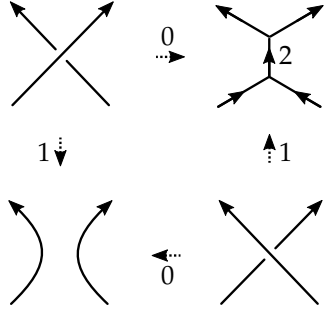
\begin{figure}[!ht]
	\labellist
	\pinlabel $1$ at 65 280
	\pinlabel $1$ at 505 280
	\pinlabel $0$ at 290 490
	\pinlabel $0$ at 290 55
	\pinlabel $2$ at 505 455
	\endlabellist
	\centering
	\includegraphics[width=.3\textwidth]{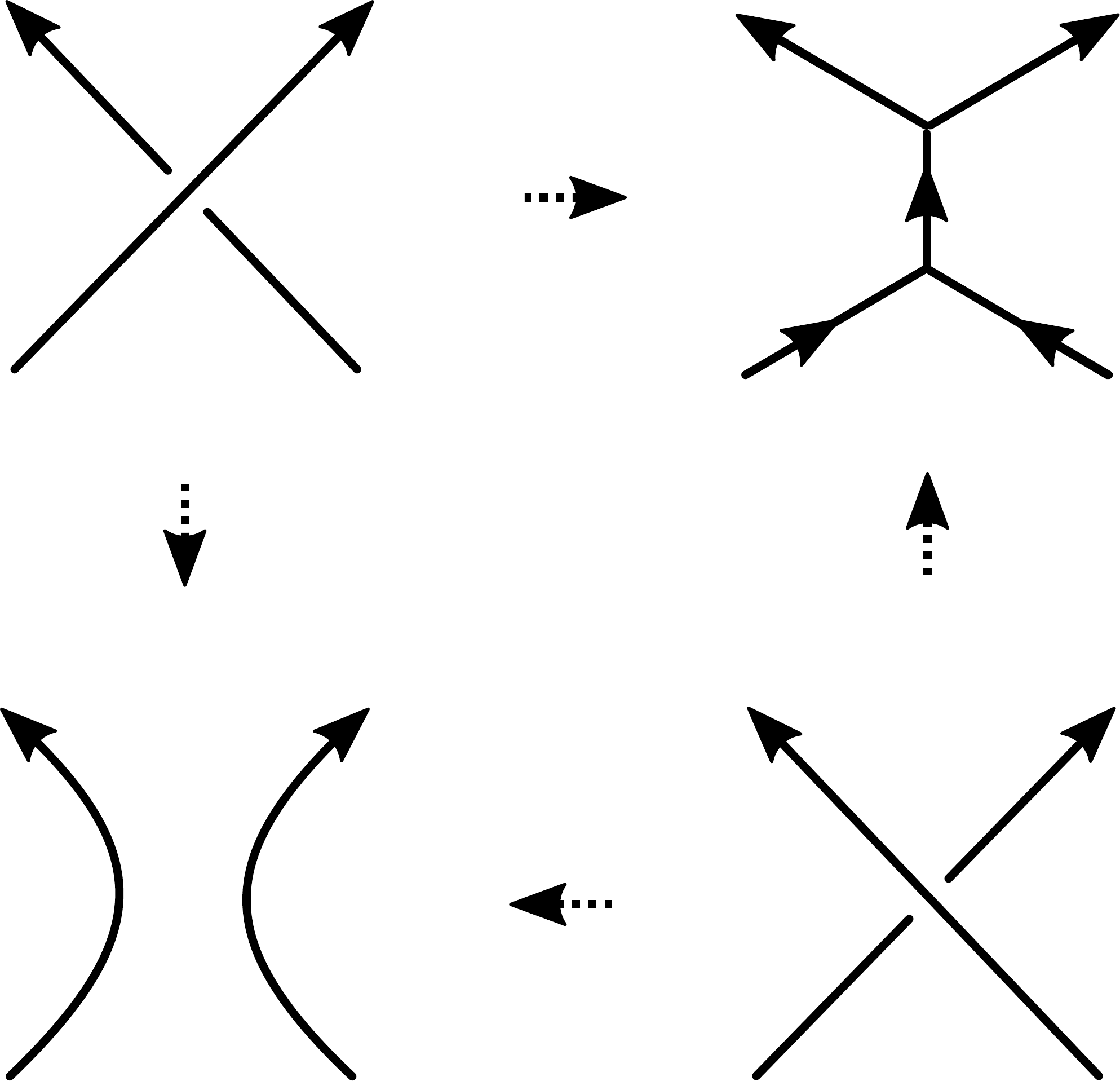}
	\captionsetup{width=.8\linewidth}
	\caption{MOY graphs obtained by resolving crossings of an oriented link diagram. An edge of an MOY graph without an explicit label is labeled $1$. The top left crossing is positive and the bottom right crossing is negative.}
	\label{fig:resolutions}
\end{figure}

We now define the differential on $\KRC_N(D;R)$. Let $v,w$ be a pair of complete resolutions for which there is a crossing $c \in \{1,\ldots,n\}$ such that $v(c) = 0$ and $w(c) = 1$ while $v(c') = w(c')$ for all $c'\neq c$. Associate to the pair $v,w$ the $\sl(N)$ foam $F_{vw}\colon D_v \to D_w$ given in Figure~\ref{fig:zipUnzip}. We note that $F_{vw}$ is a foam without dots. We refer to $\sr F_N(F_{vw};R)\colon \sr F_N(D_v;R) \to \sr F_N(D_w;R)$ as the \textit{edge map} from $v$ to $w$. The differential $d$ on $\KRC_N(D;R)$ is defined to be the signed sum of all edge maps where the edge map $\sr F_N(F_{vw};R)$ associated to $v,w$ is given the sign $(-1)^{\sum_{i < c} v(i)}$, just as in Khovanov homology \cite{MR1917056}. The grading shifts in the definition of $\KRC_N(D;R)$ ensure that this differential preserves $q$-grading and increases $h$-grading by $1$.

\begin{figure}[!ht]
	\centering
	\includegraphics[width=.3\textwidth]{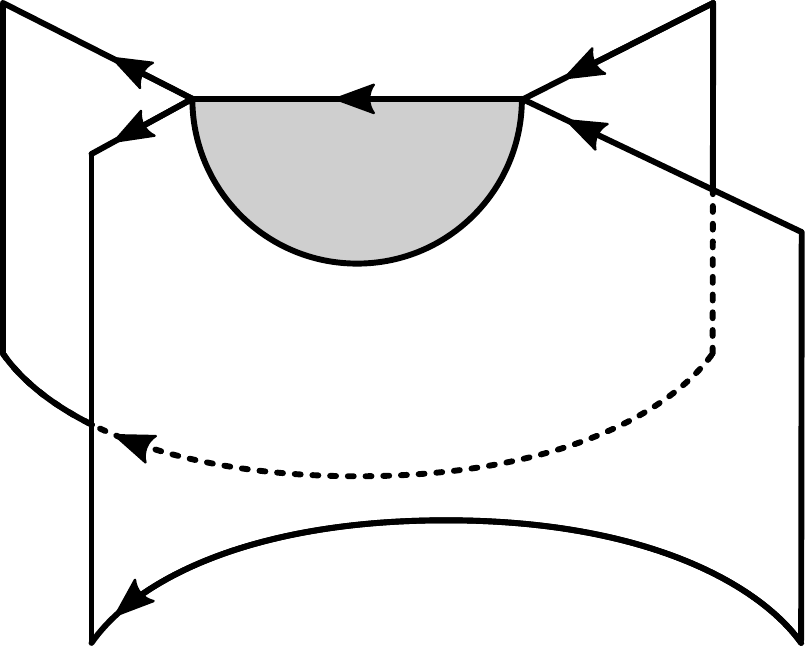}\hspace{30pt}
	\includegraphics[width=.3\textwidth]{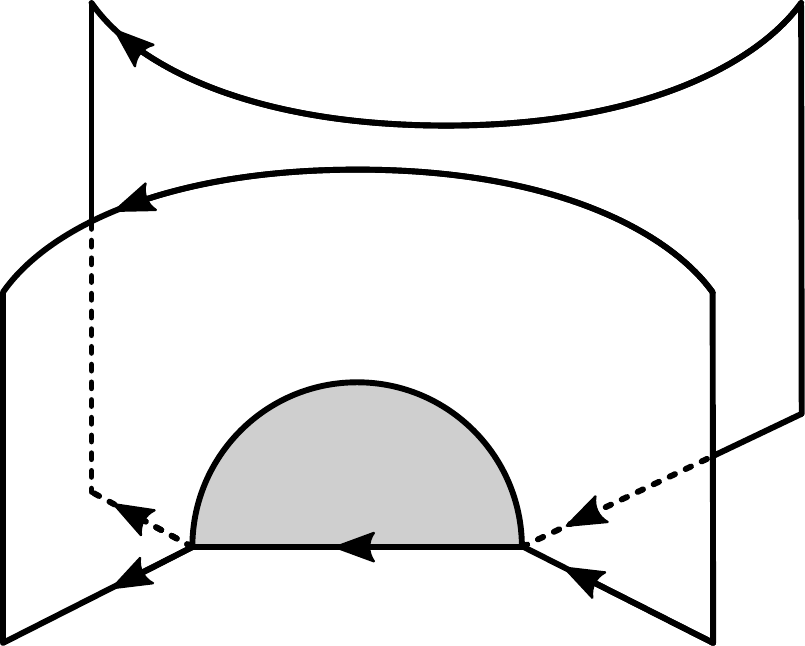}
	\captionsetup{width=.8\linewidth}
	\caption{The foam $F_{vw}$ from $D_v$ on the bottom to $D_w$ on the top. The foam on the right is for a positive crossing while the foam on the left is for a negative crossing.}
	\label{fig:zipUnzip}
\end{figure}

By Corollary~\ref{cor:FNwithRCoefficients}, there is an isomorphism of bigraded chain complexes over $R$ between $\KRC_N(D;R)$ and $\KRC_N(D;\Z) \otimes R$. The bigraded chain homotopy type of $\KRC_N(D;\Z)$ is an invariant of the oriented link $L$ that $D$ represents (see \cite[Corollary 4.4]{MR4164001}, \cite[section 4.2]{MR3545951}), so the same is true for any arbitrary coefficient ring $R$. The proof involves checking invariance under Reidemeister moves, by defining homotopy equivalences using maps induced by foams. 

The homology of $\KRC_N(D;R)$ is the \textit{(Khovanov--Rozansky) $\sl(N)$ link homology of $L$ with coefficients in $R$} and is denoted $\KR_N(L;R)$. We sometimes refer to $\KR_N(L;R)$ as \textit{unreduced} $\sl(N)$ link homology to distinguish it from the \textit{reduced} variant defined in section~\ref{subsec:BasepointOperators}. Since the $\sl(N)$ chain complex over $R$ is obtained by tensoring the complex over $\Z$ with $R$, we note that the universal coefficients theorem for $\KR_N(L;R)$ is valid. Since the $q$-graded rank of $\sr F_N(D_v;R)$ agrees with the MOY evaluation $\langle D_v\rangle$ of $D_v$ (Corollary~\ref{cor:FNwithRCoefficients}), it follows that if $R$ is a field, then the $q$-graded Euler characteristic of $\KR_N(L;R)$ coincides with the $\sl(N)$ polynomial invariant $P_N(L)$ of $L$. 

The categorification of the $\sl(N)$ link polynomial constructed by Khovanov--Rozansky in \cite{MR2391017} is defined over the field of rational numbers and uses matrix factorizations instead of foam evaluations. The construction described here is defined over an arbitrary ring and recovers Khovanov--Rozansky's invariant when taken over $\Q$ (see \cite[Proposition 4.3]{MR4164001}, \cite[Theorem 4.2]{MR3545951}).

When $N = 2$, Khovanov--Rozansky's categorification of the $\sl(2)$ link invariant coincides with Khovanov homology with rational coefficients \cite{MR2391017}. With integral (and therefore arbitrary) coefficients, Khovanov homology agrees with $\sl(2)$ link homology described here. More precisely, there is an isomorphism \[
	\KR_2^{i,j}(L;R) \cong \Kh^{i,-j}(m(L),R)
\]where $m(L)$ denotes the mirror of $L$, and $i$ and $j$ denote homological and quantum gradings, respectively. This can be seen by directly comparing the constructions, taking care in dealing with dots and signs. For the interested reader, we note that a dot on a facet labeled $1$ in our current setup basically corresponds to the negative of a dot in the usual Khovanov setting, but a dot in our setting negates when moved across a binding involving a facet labeled $2$ (see \cite[Equation (11) of Proposition 3.38]{MR4164001}).

\subsection{Basepoint operators and reduced $\sl(N)$ link homology}\label{subsec:BasepointOperators}

\begin{df}
	A \textit{basepoint} on an $\sl(N)$ MOY graph $\Gamma$ is a point on the interior of an edge of $\Gamma$. Let $p$ be a basepoint on $\Gamma$ lying on an edge $e$, and let $k = \ell(e)$. We define an operator $e_w(p)$ on $\sr F_N(\Gamma;R)$ for each weight $w = 1,\ldots,k$, called the \textit{basepoint operator of weight $w$ associated to $p$}. Let $\bo{F}\colon \emp \to \Gamma$ be a dotted foam, and define $e_w(p)\bo{F}$ to be the dotted foam obtained by adding a dot of weight $w$ to the facet $f$ of $F$ whose boundary contains $p$. In particular, $e_w(p)\bo{F} = e_w\bo{F}$ where $e_w$ is the $w$th elementary symmetric polynomial in the $k = \ell(f)$ variables associated to $f$. The basepoint operator $e_w(p)$ depends only on the edge $e$, and not on the location of $p$ within $e$. 
\end{df}
\begin{lem}\label{lem:basepointOperatorsSelfAdjoint}
	Basepoint operators are self-adjoint with respect to the bilinear pairing $\langle -,-\rangle_R$. 
\end{lem}
\begin{proof}
	If $\bo{F},\bo{G}$ are dotted foams $\emp\to\Gamma$, then \[
		\langle e_w(p)\bo{F},\bo{G}\rangle_R = \langle e_w\bo{F} \cup \widebar{\bo{G}}\rangle_R = \langle \bo{F} \cup \widebar{e_w\bo{G}}\rangle_R = \langle \bo{F},e_w(p)\bo{G}\rangle_R
	\]where the first and third equalities follow from definitions. 
	Let $f$ be the facet of $F$ adjacent to $e$ and let $g$ be the facet of $G$ adjacent to $e$. Then $f$ and $g$ are parts of a single facet of $F \cup \widebar{G}$. Since dots may move freely within a facet, we find that $e_w\bo{F} \cup \widebar{\bo{G}} = \bo{F} \cup \widebar{e_w\bo{G}}$ as dotted foams, which gives the second equality. 
\end{proof}

Lemma~\ref{lem:basepointOperatorsSelfAdjoint} is valid in both $\RFoam_N(\emp,\Gamma)$ and $\sr F_N(\Gamma;R)$. To see that $e_w(p)$ is well-defined on $\sr F_N(\Gamma;R)$, suppose $\sum_i a_i \bo{F}_i \in \RFoam_N(\emp,\Gamma)$ represents zero in $\sr F_N(\Gamma;R)$. Then for any dotted foam $\bo{G}\colon\emp \to \Gamma$, we have \[
	\sum_i a_i\langle e_w(p)\bo{F}_i, \bo{G}\rangle_R = \sum_i a_i \langle \bo{F}_i,e_w(p)\bo{G}\rangle_R = 0.
\]The basepoint operator $e_w(p)\colon \sr F_N(\Gamma;R) \to \sr F_N(\Gamma;R)$ is homogeneous of degree $2w$.

\begin{lem}\label{lem:weight1basepointOpHasNthPowerZero}
	Let $\Gamma$ be an $\sl(N)$ MOY graph, and let $X = e_1(p)$ be the weight $1$ basepoint operator associated to a basepoint $p$ lying on an edge of $\Gamma$ labeled $1$. Then $X^N = 0$. 
\end{lem}
\begin{proof}
	Let $O^1$ be an oriented circle in the plane labeled $1$ viewed as an $\sl(N)$ MOY graph, and let $X$ be the weight $1$ basepoint operator associated to any basepoint lying on $O^1$. Let $D\colon \emp \to O^1$ be a disc, viewed as a foam with a single facet labeled $1$ with no dots. We claim that the classes of $D, XD, \ldots, X^{N-1}D$ in $\sr F_N(O^1;R)$ form an $R$-basis. To see this, we apply a special case of the sphere relation of \cite[Proposition 3.32]{MR4164001}, that if $X^kS$ denotes the $2$-sphere viewed as an $\sl(N)$ foam with a single facet labeled $1$ with $k$ dots of weight $1$, then \[
		\langle X^kS\rangle_R = \begin{cases}
			-1 & k = N-1\\
			0 & \text{else.}
		\end{cases}
	\]Since $\sr F_N(O^1;R)$ is free of rank $N$, it now follows that $D,XD,\ldots,X^{N-1}D$ is a basis with dual basis $X^{N-1}D,\ldots,XD,D$. Thus $X^N = 0$ in the special case that $\Gamma = O^1$. 

	Now let $\Gamma$ be an arbitrary $\sl(N)$ MOY graph with an edge labeled $1$, and let $X$ be the associated weight $1$ basepoint operator. Let $O^1 \amalg \Gamma$ be the $\sl(N)$ MOY graph obtained from $\Gamma$ by adding an small oriented circle labeled $1$ near the basepoint $p$ in such a way that $O^1$ bounds a disc in the plane that is disjoint from $\Gamma$. We assume that the orientation on $O^1$ is chosen so that merging $O^1$ with $\Gamma$ near the basepoint defines a foam $M\colon O^1 \amalg \Gamma \to \Gamma$. 
	It is straightforward to see that $X\colon \sr F_N(\Gamma;R) \to \sr F_N(\Gamma;R)$ agrees with the composite map \[
		\begin{tikzcd}
			\sr F_N(\Gamma;R) \ar[r] & \sr F_N(O^1;R) \otimes \sr F_N(\Gamma;R) \ar[r] & \sr F_N(O^1 \amalg \Gamma;R) \ar[r,"\sr F_N(M;R)"] & \sr F_N(\Gamma;R)
		\end{tikzcd}
	\]where the first map is $a \mapsto XD \otimes a$ and the second map is the isomorphism of Corollary~\ref{cor:splitUnionFormulaForFN}. Since $X^N = 0$ on $\sr F_N(O^1;R)$, it follows that $X^N = 0$ on $\sr F_N(\Gamma;R)$ as well.
\end{proof}

Let $D$ be a diagram of an oriented link $L$ with its crossings ordered, and let $\KRC_N(D;R)$ be its $\sl(N)$ chain complex. A \textit{basepoint} on $D$ is a point on $D$ away from the crossings. If $p$ is a basepoint on $D$, then for any complete resolution $v$ of $D$, the same point $p$ may be viewed as a basepoint on the $\sl(N)$ MOY graph $D_v$. The basepoint $p$ on $D_v$ lies on an edge labeled $1$, so there is an associated weight $1$ basepoint operator $e_1(p)\colon \sr F_N(D_v;R) \to \sr F_N(D_v;R)$. 

\begin{df}
	The \textit{basepoint operator on $\KRC_N(D;R)$ with respect to $p$} is the map \[
		X_p\colon \KRC_N(D;R) \to \KRC_N(D;R)
	\]given by the direct sum of the weight $1$ basepoint operators $e_1(p)\colon \sr F_N(D_v;R) \to \sr F_N(D_v;R)$. 
\end{df}

It follows from Lemma~\ref{lem:weight1basepointOpHasNthPowerZero} that $X_p^N = 0$. It is clear that $X_p$ commutes with the differential on $\KRC_N(D;R)$ because each edge map commutes with the relevant weight $1$ basepoint operators, so it induces a map on $\sl(N)$ link homology. If $p,q$ are two basepoints that lie on the same component of $L$, then it turns out that $X_p$ and $X_q$ are chain homotopic and therefore induce the same map on $\KR_N(L;R)$. One may see this by adapting the matrix factorization argument of \cite[Lemma 5.16]{MR3447099} to this setting by explicitly defining a chain homotopy $H$ using foams. 

\begin{df}
	The \textit{reduced $\sl(N)$ link homology} $\rKR_N(L,p;R)$ of $L$ with respect to the basepoint $p$ is the homology of the complex\[
		q^{1-N}(X_p^{N-1}\KRC_N(D;R))
	\]where $q^{1-N}$ denotes a shift in quantum grading. This shift in grading is made so that the reduced $\sl(N)$ link homology of the unknot is concentrated in quantum grading $0$.
\end{df}

Invariance of $\rKR_N(L,p;R)$ can be shown using the argument of \cite[section 3]{MR2034399}. The choice of another basepoint $q$ on $D$ defines another basepoint operator $X_q$ on $\rKR_N(L,p;R)$ for which $X_q^N = 0$. 

We record here a split union formula for reduced $\sl(N)$ link homology, which follows from Corollary~\ref{cor:splitUnionFormulaForFN}. 

\begin{cor}\label{cor:splitUnionReducedslNHomology}
	Let $L,J$ be oriented links, and let $L\amalg J$ denote their split union. If $p$ is a basepoint on $L$, then there is an isomorphism \[
		\rKR_N(L\amalg J,p;R) \cong \rKR_N(L,p;R) \otimes \KR_N(J;R)
	\]
\end{cor}
\begin{proof}
	We fix a split diagram for $L \amalg J$, so that every MOY graph obtained from a complete resolution is the split union of MOY graphs obtained from diagrams of $L$ and $J$ individually. The result now follows from Corollary~\ref{cor:splitUnionFormulaForFN}, together with the fact that the isomorphism in the statement is natural under maps induced by split foams. 
\end{proof}

\section{The operator $\nabla$}\label{sec:ConstructionOfTheOp}

In this section, we construct the operator $\nabla$ on $\sl(N)$ link homology $\KR_N(L;R)$ when the characteristic of $R$ divides $N$. This operator is induced by a chain map on the chain complex $\KRC_N(D;R)$ associated to a diagram $D$ of the link. This chain map is defined to be the direct sum of operators that we define on each $\sr F_N(D_v;R)$. We begin with the construction of the operator on $\sr F_N(\Gamma;R)$ for any $\sl(N)$ MOY graph $\Gamma$. By abuse of notation, we will refer to the operator as $\nabla$ at every level of the construction. 

Let $R$ be a ring. We will later make an explicit assumption that the characteristic of $R$ divides $N$. 

\begin{df}
	Let $\nabla\colon R[X_1,\ldots,X_k] \to R[X_1,\ldots,X_k]$ be the $R$-linear map given by taking the sum over all $k$ first-order partial derivatives, which are taken formally. On monomials, the operator is given by \[
		\nabla(X_1^{d_1}\cdots X_k^{d_k}) = \sum_{i=1}^k d_i\cdot X_1^{d_1}\cdots X_i^{d_i - 1} \cdots X_k^{d_k}.
	\]
\end{df}

Observe that the operator $\nabla$ satisfies the Leibniz rule $\nabla(PQ) = \nabla(P)Q + P\nabla(Q)$, and its $n$-fold composite $\nabla^n$ satisfies the generalized Leibniz rule \[
	\nabla^n(PQ) = \sum_{i=0}^n \binom{n}{i} \nabla^{n-i}(P)\nabla^i(Q).
\]The following lemma highlights a feature of $\nabla$ that we will use a number of times. 

\begin{lem}\label{lem:identifyingVariablesCommutesWithD}
	Suppose $s\colon I \to J$ is a map between finite sets. If $\phi_s\colon R[X_i\:|\: i\in I] \to R[Y_j\:|\: j \in J]$ is the $R$-algebra map given by $\phi_s(X_i) = Y_{s(i)}$, then $\nabla \circ \phi_s = \phi_s \circ \nabla$. 
\end{lem}
\begin{proof}
	By linearity, it suffices to verify the formula for monomials. Furthermore, it is clear that $\nabla\circ\phi_s = \phi_s\circ \nabla$ is valid on each individual indeterminant $X_i$ for $i \in I$. The result then follows by induction using the Leibniz rule. Any monomial of degree at least two may be written as a product of monomials $P,Q$ of strictly lower degree, so \begin{align*}
		\nabla(\phi_s(PQ)) &= \nabla(\phi_s(P))\phi_s(Q) + \phi_s(P)\nabla(\phi_s(Q))\\
		&= \phi_s(\nabla(P))\phi_s(Q) + \phi_s(P)\phi_s(\nabla(Q)) = \phi_s(\nabla(PQ)).\qedhere
	\end{align*}
\end{proof}

\begin{df}\label{df:opDonRFoamN}
	Let $\bo{F} = PF$ be a dotted foam from $\Gamma_0$ to $\Gamma_1$, where $F$ is the underlying foam without dots and $P \in R[F]$. Let $\nabla\colon \RFoam_N(\Gamma_0,\Gamma_1) \to \RFoam_N(\Gamma_0,\Gamma_1)$ be the $R$-linear map defined by\[
		\nabla(PF) = \nabla(P) F.
	\]
\end{df}

\begin{rem}
	If $Q$ is an arbitrary polynomial in $R[F]$, then $\nabla(QF) = \nabla(Q)F$. The identity is the definition of $\nabla$ when $Q$ arises from a set of dots, and follows in general by linearity. Furthermore, the Leibniz rule for $\nabla$ implies that $\nabla(Q\bo{F}) = \nabla(Q)\bo{F} + Q\nabla(\bo{F})$. We also note that $\nabla$ is homogeneous of degree $-2$. 
\end{rem}

\begin{lem}\label{lem:LeibnizForComposition}
	If $\bo{F}\colon \Gamma_0 \to \Gamma_1$ and $\bo{G}\colon \Gamma_1 \to \Gamma_2$ are dotted foams, then \[
		\nabla(\bo{F}\cup \bo{G}) = \nabla(\bo{F}) \cup \bo{G} + \bo{F} \cup \nabla(\bo{G}).
	\]
\end{lem}
\begin{proof}
	Write $\bo{F} = PF$ and $\bo{G} = QG$ where $P \in R[F]$ and $Q \in R[G]$. There is a canonical map $\phi\colon R[F] \otimes R[G] \to R[F \cup G]$ induced by the gluing. If $f$ is a facet of either $F$ or $G$, then $f$ is a portion of some facet $h$ of $F \cup G$, where $\ell(f) = \ell(h)$. The map $\phi$ sends each symmetric polynomial in the variables $(X_f)_1,\ldots,(X_f)_{\ell(f)}$ to the corresponding symmetric polynomial in the variables $(X_h)_1,\ldots,(X_h)_{\ell(h)}$. With this notation, we have $\bo{F} \cup \bo{G} = \phi(P \otimes Q)(F \cup G)$. 

	Note that $R[F] \otimes R[G]$ is a subring of a polynomial ring, and that it is preserved under $\nabla$. By Lemma~\ref{lem:identifyingVariablesCommutesWithD}, we have that $\nabla$ commutes with $\phi$. Thus \[
		\nabla(\phi(P \otimes Q)) = \phi(\nabla(P \otimes Q)) = \phi(\nabla(P) \otimes Q) + \phi(P \otimes \nabla(Q))
	\]as required. 
\end{proof}

For the following lemma, we temporarily let $R = \Z$ so that if $\bo{F}$ is a dotted foam, then $\nabla\bo{F}$ is a formal $\Z$-linear combination of dotted foams. 
The Robert--Wagner evaluation of a $\Z$-linear combination of dotted foams is defined by extending the evaluation $\Z$-linearly. 

\begin{lem}\label{lem:DcommutesWithEvaluation}
	Let $\bo{F}$ be a closed dotted $\sl(N)$ foam. Then \[
		\langle \nabla\bo{F} \rangle = \nabla\langle \bo{F}\rangle.
	\]
\end{lem}
\begin{proof}
	We first claim that \[
		\nabla\langle\bo{F}\rangle = \sum_c (-1)^{s(F,c)} \frac{\nabla(P(\bo{F},c))}{Q(F,c)}.
	\]Let $W$ be a product of terms of the form $X_i - X_j$ for $i < j$ for which $W/Q(F,c)$ is a polynomial for each $c$. Such a polynomial $W$ exists by definition of $Q(F,c)$. Then we clear denominators with $W$ to obtain \[
		W\langle\bo{F}\rangle = \sum_c (-1)^{s(F,c)}P(\bo{F},c)\frac{W}{Q(F,c)}.
	\]Both $W$ and $W/Q(F,c)$ are products of terms of the form $X_i - X_j$. Since $\nabla(X_i - X_j) = 0$, it follows that $\nabla$ commutes with multiplication by $W$ and $W/Q(F,c)$ by the Leibniz rule. Thus \[
		W\nabla\langle\bo{F}\rangle = \sum_c (-1)^{s(F,c)} \nabla(P(\bo{F},c)) \frac{W}{Q(F,c)}
	\]which proves the claim.

	It suffices to prove that \[
		\langle \nabla\bo{F}\rangle = \sum_c (-1)^{s(F,c)} \frac{\nabla(P(\bo{F},c))}{Q(F,c)}.
	\]If $S \in \Z[F]$, then \[
		\langle SF\rangle = \sum_c (-1)^{s(F,c)} \frac{P(SF,c)}{Q(F,c)}
	\]where $P(SF,c)$ is obtained from $S$ by replacing the formal variables $(X_f)_1,\ldots,(X_f)_{\ell(f)}$ with the variables $X_i$ for $i \in c(f)$. This identity is true by definition if $S$ arises from a dotted foam, and is straightforward to verify in general by linearity using the fact that $s(F,c)$ and $Q(F,c)$ depend only on the underlying foam and the coloring. If $\bo{F} = SF$, then the result follows from the identity $\nabla(P(SF,c)) = P(\nabla(S)F,c)$ which in turn follows from Lemma~\ref{lem:identifyingVariablesCommutesWithD}.
\end{proof}

\begin{lem}\label{lem:evaluationOfDFisZeroModN}
	If the characteristic of $R$ divides $N$, then \[
		\langle \nabla\bo{F}\rangle_R = 0
	\]for every closed dotted $\sl(N)$ foam $\bo{F}$. 
\end{lem}
\begin{proof}
	By Lemma~\ref{lem:DcommutesWithEvaluation}, $\langle \nabla\bo{F} \rangle_R$ is the image of $\nabla\langle \bo{F}\rangle$ under the map $\Z[X_1,\ldots,X_N] \to R$ given by sending each $X_i$ to $0$. Since $\nabla\langle\bo{F}\rangle$ is homogeneous, it suffices to consider the case that the ordinary degree of the polynomial $\langle\bo{F}\rangle$ is $1$ (equivalently, the quantum degree is $2$). In this case, it follows that \[
		\langle\bo{F}\rangle = m(X_1 + \cdots + X_N)
	\]for some $m \in \Z$ because $\langle\bo{F}\rangle$ is symmetric. Thus $\nabla\langle\bo{F}\rangle = mN$ by definition of $\nabla$. Because the characteristic of $R$ divides $N$, the image of $\nabla\langle\bo{F}\rangle$ in $R$ is zero. 
\end{proof}

\begin{prop}\label{prop:DwellDefinedModN}
	If the characteristic of $R$ divides $N$, then $\nabla\colon \RFoam_N(\emp,\Gamma) \to \RFoam(\emp,\Gamma)$ descends to a well-defined operator $\nabla\colon \sr F_N(\Gamma;R) \to \sr F_N(\Gamma;R)$. 
\end{prop}
\begin{proof}
	We use ``integration by parts''. Suppose $\sum_i a_i \bo{F}_i \in \RFoam_N(\emp,\Gamma)$ represents zero in $\sr F_N(\Gamma;R)$. We must show that $\sum_i a_i \nabla(\bo{F}_i)$ also represents zero in $\sr F_N(\Gamma;R)$. Let $\bo{G}\colon \Gamma\to\emp$ be an arbitrary dotted foam, and observe that \[
		0 = \sum_i a_i \langle \nabla(\bo{F}_i \cup \bo{G})\rangle_R = \sum_i a_i \langle \nabla(\bo{F}_i) \cup \bo{G}\rangle_R + \sum_i a_i \langle \bo{F}_i \cup \nabla(\bo{G})\rangle_R
	\]where the first equality follows from Lemma~\ref{lem:evaluationOfDFisZeroModN} and the second equality is the Leibniz rule (Lemma~\ref{lem:LeibnizForComposition}). The second sum on the right-hand side vanishes by assumption that $\sum_i a_i \bo{F}_i$ represents zero in $\sr F_N(\Gamma;R)$, so the first sum vanishes as well.
\end{proof}

\begin{lem}\label{lem:FGCommutesWithD}
	If $\bo{G}\colon \Gamma_0 \to \Gamma_1$ is a dotted foam for which $\nabla(\bo{G}) = 0$, then $\sr F_N(\bo{G};R)\colon \sr F_N(\Gamma_0;R) \to \sr F_N(\Gamma_1;R)$ commutes with $\nabla$. In particular, if $\bo{G}$ has no dots, then $\sr F_N(\bo{G};R)$ commutes with $\nabla$.
\end{lem}
\begin{proof}
	The result follows from the Leibniz rule (Lemma~\ref{lem:LeibnizForComposition}).
\end{proof}

\vspace{10pt}

Let $D$ be a diagram of an oriented link $L$ with its crossings ordered, and let $\KRC_N(D;R)$ be its $\sl(N)$ chain complex. 

\begin{df}
	Let $R$ be a ring whose characteristic divides $N$. Let \[
		\nabla\colon \KRC_N(D;R) \to \KRC_N(D;R)
	\]be the direct sum of the maps $\nabla\colon \sr F_N(D_v;R) \to \sr F_N(D_v;R)$ over all complete resolutions $v$. We note that $\nabla$ preserves homological grading and decreases quantum grading by $2$.
\end{df}

The differential of $\KRC_N(D;R)$ is given by the signed sum of maps induced by foams without dots, so $\nabla\colon \KRC_N(D;R) \to \KRC_N(D;R)$ is a chain map by Lemma~\ref{lem:FGCommutesWithD}. It follows that $\nabla$ descends to a map on $\sl(N)$ link homology. The map $\nabla\colon \KR_N(L;R) \to \KR_N(L;R)$ does not depend on the choice of diagram, though we will not make use of this fact. The proof stems from the observation that the homotopy equivalences defined between chain complexes associated to diagrams differing by Reidemeister moves are all induced by maps of foams without dots (see for example the proof of invariance in \cite[section 7]{MR2491657}). 

\section{Proof of the main theorem}\label{sec:ReducedslNLinkHomology}

\begin{lem}\label{lem:equivarianceNprime}
	Let $\Gamma$ be an $\sl(N)$ MOY graph, and let $X,Y,Z$ be weight $1$ basepoint operators associated to three basepoints lying on edges of $\Gamma$ that are labeled $1$. Assume that the characteristic of $R$ divides $N$ so that the operator $\nabla$ on $\sr F_N(\Gamma;R)$ is defined. Then the map \[
		X^{N-1}\nabla^{N-1}\colon Y^{N-1}\sr F_N(\Gamma;R) \to X^{N-1}\sr F_N(\Gamma;R)
	\]intertwines the action of $Z$ on $Y^{N-1}\sr F_N(\Gamma;R)$ with the action of $Z-Y$ on $X^{N-1}\sr F_N(\Gamma;R)$, which is to say that\[
		(X^{N-1}\nabla^{N-1})\circ Z = (Z - Y) \circ (X^{N-1}\nabla^{N-1}).
	\]
\end{lem}

In particular, if $Z = X$, then $X^{N-1}\nabla^{N-1}$ intertwines the action of $X$ on $Y^{N-1}\sr F_N(\Gamma;R)$ with the action of $-Y$ on $X^{N-1}\sr F_N(\Gamma;R)$ because $X^N = 0$ by Lemma~\ref{lem:weight1basepointOpHasNthPowerZero}.

\begin{proof}
	Set $M = N -1$, and observe that for any dotted foam $\bo{F}\colon \emp \to \Gamma$, \begin{align*}
		&\mathrel{\phantom{=}} X^M(\nabla^M(ZY^M\bo{F}) + (Y - Z)\nabla^M(Y^M\bo{F}))\\
		&= X^M\sum_{i=0}^M\binom{M}{i}(\nabla^{M-i}(ZY^M) + (Y-Z)\nabla^{M-i}(Y^M))\nabla^i(\bo{F})\\
		&= X^M\sum_{i=0}^M\binom{M}{i}((M-i)\nabla^{M-i-1}(Y^M) + Y\nabla^{M-i}(Y^M))\nabla^i(\bo{F})\\
		&= X^M\sum_{i=0}^M\binom{M}{i}\left((M-i)\frac{M!}{(i+1)!} + \frac{M!}{i!} \right)Y^{i+1}\nabla^i(\bo{F})
	\end{align*}by the generalized Leibniz rule. Now note that \[
		(M-i)\frac{M!}{(i+1)!} + \frac{M!}{i!} = ((M-i) + (i+1))\frac{M!}{(i+1)!} \equiv 0 \bmod N.\qedhere
	\]
\end{proof}

Proposition~\ref{prop:doubleCompositesNprime} stated in the introduction follows from the following proposition. 

\begin{prop}\label{prop:inversesNprime}
	Let $\Gamma$ be an $\sl(N)$ MOY graph, and let $X,Y$ be weight $1$ basepoint operators associated to two basepoints lying on edges of $\Gamma$ that are labeled $1$. Assume that the characteristic of $R$ divides $N$ so that the operator $\nabla$ on $\sr F_N(\Gamma;R)$ is defined, and consider the $R$-linear maps \[
		\begin{tikzcd}
			X^{N-1}\sr F_N(\Gamma;R) \ar[r,swap,bend right = 40pt,"Y^{N-1}\nabla^{N-1}"] & Y^{N-1}\sr F_N(\Gamma;R). \ar[l,swap,bend right=40pt,"X^{N-1}\nabla^{N-1}"]
		\end{tikzcd}
	\]
	Then their composition is given by \[
		(X^{N-1}\nabla^{N-1})\circ(Y^{N-1}\nabla^{N-1}) = \begin{cases}
			\Id & N \text{ is prime}\\
			0 & N \text{ is composite}.
		\end{cases}
	\]
\end{prop}
\begin{proof}
	Set $M = N-1$. For any dotted foam $\bo{F}\colon \emp\to\Gamma$, we have \begin{align*}
		Y^M\nabla^M(X^M\bo{F}) &= Y^M\sum_{i=0}^M \binom{M}{i} \nabla^{M-i}(X^M)\nabla^i(\bo{F}) = \sum_{i=0}^M\binom{M}{i}\frac{M!}{i!}Y^MX^i\nabla^i(\bo{F}).
	\end{align*}Using the generalized Leibniz rule and the identity $X^{M+1} = X^N = 0$ of Lemma~\ref{lem:weight1basepointOpHasNthPowerZero}, we compute \begin{align*}
		X^M\nabla^M(Y^M\nabla^M(X^M\bo{F})) &= X^M \sum_{i=0}^M \binom{M}{i}\frac{M!}{i!} \nabla^M(Y^MX^i\nabla^i(\bo{F}))\\
		&= X^M \sum_{i=0}^M \binom{M}{i}\frac{M!}{i!} \binom{M}{i} \nabla^i(X^i)\nabla^{M-i}(Y^M\nabla^i(\bo{F}))\\
		&= X^M M! \sum_{i=0}^M \binom{M}{i}\binom{M}{i} \sum_{j=0}^{M-i} \binom{M-i}{j} \nabla^{M-i-j}(Y^M)\nabla^{i+j}(\bo{F})\\
		&= X^M M! \sum_{i=0}^M \sum_{j=0}^{M-i} \binom{M}{i}\binom{M}{i+j}\binom{i+j}{i}\frac{M!}{(i+j)!} Y^{i+j}\nabla^{i+j}(\bo{F})
	\end{align*}where the last equality uses the identity \[
		\binom{M}{i}\binom{M-i}{j} = \binom{M}{i+j}\binom{i+j}{i}.
	\]The double sum is over all pairs of nonnegative integers $i,j$ for which $i + j \leq M$. We rewrite this sum using $\ell = i + j$ to obtain \begin{align*}
		X^M\nabla^M(Y^M\nabla^M(X^M\bo{F})) &= \sum_{\ell=0}^M M!\binom{M}{\ell}\frac{M!}{\ell!} \left(\sum_{i=0}^{\ell} \binom{M}{i} \binom{\ell}{i} \right) X^M Y^\ell \nabla^\ell(\bo{F}).
	\end{align*}

	Observe that the $\ell = 0$ term in the sum is \[
		M!M!\cdot X^M\bo{F} = \begin{cases}
			X^M\bo{F} & N \text{ is prime}\\
			0 & N \text{ is composite.}
		\end{cases}
	\]It suffices to show that \[
		M!\binom{M}{\ell}\frac{M!}{\ell!}\sum_{i=0}^\ell \binom{M}{i}\binom{\ell}{i} \equiv 0 \bmod N
	\]for $1 \leq \ell \leq M$. By Wilson's theorem, if $N$ is composite and $N \neq 4$, then $M!\equiv 0\bmod N$. If $N = 4$, then the result is true by direct computation for $\ell = 1,2,3$. Finally, if $N$ is prime, then for $i = 0,\ldots,M$, \[
		\binom{M}{i} = \frac{M(M-1)\cdots(M-i+1)}{i!} \equiv \frac{(-1)(-2)\cdots(-i)}{i!} \equiv (-1)^i \bmod N
	\]because $i!$ is invertible mod $N$. It follows that \[
		\sum_{i=0}^\ell \binom{M}{i}\binom{\ell}{i} \equiv \sum_{i=0}^\ell (-1)^i \binom{\ell}{i} \equiv 0 \bmod N.\qedhere
	\]
\end{proof}

\begin{proof}[Proof of Theorem~\ref{thm:mainThm}]
	\leavevmode 
	\begin{enumerate}
		\item Let $D$ be a diagram of an oriented link $L$ with basepoints $q,r$. If $P$ is prime and the characteristic of $R$ is $P$, then by Proposition~\ref{prop:doubleCompositesNprime} (Proposition~\ref{prop:inversesNprime}), the chain maps \[
		\begin{tikzcd}
			X_q^{P-1}\KRC_P(D;R) \ar[r,swap,bend right = 40pt,"X_r^{P-1}\nabla^{P-1}"] & X_r^{P-1}\KRC_P(D;R) \ar[l,swap,bend right=40pt,"X_q^{P-1}\nabla^{P-1}"]
		\end{tikzcd}
		\]are inverse isomorphisms. Let $\Phi\colon \rKR_P(L,q;R) \to \rKR_P(L,r;R)$ be the isomorphism on homology induced by $X_r^{P-1}\nabla^{P-1}$. By construction, $\Phi$ preserves both gradings. The identity $\Phi\circ X_q + X_r \circ \Phi = 0$ follows from Lemma~\ref{lem:equivarianceNprime}. 
		\item We use the argument appearing in the proof of \cite[Proposition 1.7]{MR3071132}. Let $L$ be an oriented link with a basepoint $q$, and consider the split union $L \amalg O$ where $O$ is an unknot. Let $r$ be a basepoint on $O$. Because reduced $\sl(P)$ link homology over $R$ is basepoint-independent, there is an isomorphism \[
			\Phi\colon \rKR_P(L \amalg O, r;R) \to \rKR_P(L\amalg O,q;R)
		\]that satisfies $\Phi\circ X_q + X_r\circ\Phi = 0$. By Corollary~\ref{cor:splitUnionReducedslNHomology}, we have an identification \[
			\rKR_P(L\amalg O,r;R) \cong \KR_P(L;R) \otimes \rKR_P(O,r;R) \cong \KR_P(L;R)
		\]because $\rKR_P(O,r;R) \cong R$. Another application of Corollary~\ref{cor:splitUnionReducedslNHomology} gives an identification \[
			\rKR_P(L \amalg O,q;R)\cong \rKR_P(L,q;R) \otimes \KR_P(O;R) \cong \rKR_P(L,q;R) \otimes R[X]/X^P
		\]because $\KR_P(O;R) \cong R[X]/X^P$. These identifications respect the basepoint actions. \qedhere
	\end{enumerate}
\end{proof}

\section{The operator on the cohomology of the Grassmannian}\label{sec:grassmannian}

Let $O^k$ denote an oriented circle in the plane with the label $k$ where $1 \leq k \leq N-1$, viewed as an $\sl(N)$ MOY graph. Let $\G(k,N)$ be the Grassmannian of $k$-dimensional complex linear subspaces of $\C^N$. There is an identification between $\sr F_N(O^k;R)$ and the cohomology of $\G(k,N)$ with a grading shift, which we explain below. If the characteristic of $R$ divides $N$ so that $\nabla$ is defined on $\sr F_N(O^k;R)$, then under the described identification, we may view $\nabla$ as an operator on $H^*(\G(k,N);R)$. In this section, we explicitly compute $\nabla$ on $H^*(\G(k,N);R)$.

We recall some facts surrounding the cohomology of $\G(k,N)$. See \cite[chapter 9.4]{MR1464693} as a reference.

\begin{df}
	A \textit{partition} $\lambda$ is a sequence $\lambda = (\lambda_1,\ldots,\lambda_\ell)$ of positive integers for which $\lambda_1 \ge \cdots \ge \lambda_\ell > 0$. The numbers $\lambda_1,\ldots,\lambda_\ell$ are called the \textit{parts} of $\lambda$, and we say that $\lambda$ is a \textit{partition of} $|\lambda|\coloneq \lambda_1 + \cdots + \lambda_\ell$. By convention, we set $\lambda_j = 0$ if $j > \ell$. 

	Let $\lambda$ be a partition with at most $k$ parts. The \textit{($k$-variable) Schur polynomial} $s_\lambda$ \textit{associated to} $\lambda$ is the integral homogeneous symmetric polynomial in $k$ variables of degree $|\lambda|$ given by the formula \[
		s_\lambda(X_1,\ldots,X_k) = \frac{\det(X_i^{\lambda_j + k - j})}{\prod_{1 \leq i < j \leq k} \: (X_i - X_j)}
	\]where $\det(X_i^{\lambda_j + k - j})$ is the determinant of the $k \x k$ matrix whose $(i,j)$-entry is $X_i^{\lambda_j + k - j}$. Given a fixed nonnegative integer $d$, a basis for the free $\Z$-module of homogeneous symmetric polynomials in $k$ variables of degree $d$ is given by the collection of Schur polynomials $s_\lambda$ where $\lambda$ varies over all partitions of $d$ with at most $k$ parts. The $i$th elementary symmetric polynomial $e_i(X_1,\ldots,X_k)$ is equal to the Schur polynomial associated to the partition of $i$ that has $i$ parts all equal to $1$. 

	Let $V = (V_1,\ldots,V_N)$ be a sequence of vector subspaces $V_1 \subset V_2 \subset \cdots \subset V_N = \C^N$ such that $\dim V_i = i$. The \textit{Schubert cycle} $\sigma_\lambda(V)$ \textit{associated to a partition $\lambda$ with at most $k$ parts relative to $V$} is the subvariety of $\G(k,N)$ given by \[
		\sigma_\lambda(V) = \{\:\Lambda_k \subset \C^N \:|\: \dim (\Lambda_k \cap V_{N-k + i - \lambda_i}) \ge i \:\}.
	\]The homology class $\sigma_\lambda \coloneq [\sigma_\lambda(V)]$ is independent of the choice of $V$. If $\lambda$ has a part strictly greater than $N - k$, then $\sigma_\lambda = 0$. A basis for the homology of $\G(k,N)$ is given by the classes $\sigma_\lambda$ associated to partitions $\lambda$ where no part of $\lambda$ exceeds $N - k$. The Schubert cycle $\sigma_\lambda(V)$ is of complex codimension $|\lambda|$ so the Poincar\'e dual $\PD(\sigma_\lambda)$ of $\sigma_\lambda$ is a cohomology class of degree $2|\lambda|$. 
\end{df}

If $\lambda$ and $\mu$ are partitions with at most $k$ parts, then $s_\lambda\cdot s_\mu$ is a homogeneous symmetric polynomial of degree $|\lambda| + |\mu|$. Because Schur polynomials form a basis for the free module of symmetric polynomials, we may write $s_\lambda\cdot s_\mu = \sum_\nu c_{\lambda\mu}^\nu s_\nu$. The numbers $c_{\lambda\mu}^\nu \in \Z$ are called the \textit{Littlewood-Richardson coefficients}. We note that there is a combinatorial formula for these coefficients called the Littlewood-Richardson rule (see \cite{MR1862150} for a survey).

\begin{prop}\label{prop:cupProductLittlewoodRichardson}
	The cup product structure on the cohomology of $\G(k,N)$ is given by multiplication of Schur polynomials. If $\lambda,\mu$ are partitions with at most $k$ parts, then \[
		\PD(\sigma_\lambda) \cup \PD(\sigma_\mu) = \sum_\nu c_{\lambda\mu}^\nu\: \PD(\sigma_\nu)
	\]where $c_{\lambda\mu}^\nu$ are the Littlewood-Richardson coefficients.
\end{prop}

For any ring $R$, the Schur polynomials, thought of as elements of $R[X_1,\ldots,X_k]^{\fk S_k}$, form an $R$-basis for the free $R$-module of symmetric polynomials. Furthermore, the cup product of the classes $\PD(\sigma_\lambda)$ and $\PD(\sigma_\mu)$ in $H^*(\G(k,N);R)$ is given by the Littlewood-Richardson coefficients, viewed as elements of $R$, just as in Proposition~\ref{prop:cupProductLittlewoodRichardson}.

Let $D\colon \emp \to O^k$ be the foam consisting of a disc labeled $k$ with no dots. For each partition $\lambda$ with at most $k$ parts, the Schur polynomial $s_\lambda$ is a symmetric polynomial in $k$ variables, so we may view $s_\lambda D$ as a linear combination of dotted foams from $\emp$ to $O^k$. 

\begin{prop}\label{prop:KRNofUnknotIsCohOfGkN}
	There is an isomorphism \[
		H^*(\G(k,N);R) \to \sr F_N(O^k;R)
	\]given by sending $\PD(\sigma_\lambda)$ to the class of $s_\lambda D$ in $\sr F_N(O^k;R)$ for each partition $\lambda$ with at most $k$ parts. This isomorphism is homogeneous of degree $-k(N-k)$.
\end{prop}
\begin{proof}
	Let $S$ denote the $2$-sphere viewed as an $\sl(N)$ foam with a single facet labeled $k$. For any partition $\lambda$ with at most $k$ parts, we may view $s_\lambda S$ as a $\Z$-linear combination of dotted closed foams. Robert--Wagner \cite[Proposition 3.38]{MR4164001} show that \[
		\langle s_\lambda S \rangle_\Z = \begin{cases}
			(-1)^{k(k+1)/2} & \lambda = (N-k,\ldots,N-k) \text{ with $k$ parts}\\
			0 & \text{else.}
		\end{cases}
	\] 
	Note that since $\sigma_{(N-k,\ldots,N-k)}$ is the homology class of a point, $\PD(\sigma_{(N-k,\ldots,N-k)})$ generates the top-degree cohomology group of $\G(k,N)$. It is now straightforward to show that the map in the statement of the proposition is an isomorphism using the fact that the Poincar\'e pairing on the cohomology of $\G(k,N)$ is perfect and the fact that $H^*(\G(k,N);R)$ and $\sr F_N(O^k;R)$ have the same $q$-graded rank. 
\end{proof}

Let $s_\lambda$ be the Schur polynomial associated to a partition $\lambda$ with at most $k$ parts. Then $\nabla(s_\lambda)$ is another homogeneous symmetric polynomial in $k$ variables and may therefore be written as a linear combination  \[
	\nabla(s_\lambda) = \sum_\nu d_\lambda^\nu s_\nu.
\]Since $\nabla$ is well-defined on $\sr F_N(O^k;R)$ when the characteristic of $R$ divides $N$, Proposition~\ref{prop:KRNofUnknotIsCohOfGkN} implies the following observation about the cohomology ring of $\G(k,N)$. 

\begin{cor}\label{cor:DdefinedOnCohOfGkN}
	If the characteristic of $R$ divides $N$, then there is a map \[
		\nabla\colon H^*(\G(k,N);R) \to H^{*-2}(\G(k,N);R)
	\]given by $\PD(\sigma_\lambda) \mapsto \sum_\nu d_\lambda^\nu \PD(\sigma_\nu)$ for every partition $\lambda$ with at most $k$ parts that satisfies the Leibniz rule $\nabla(\alpha \cup \beta) = \nabla(\alpha) \cup \beta + \alpha \cup \nabla(\beta)$ with respect to the cup product. 
\end{cor}

Suppose the characteristic of $R$ does not divide $N$. Then it is easy to find a partition $\lambda$ for which $\lambda_1 > N - k$, implying that $\PD(\sigma_\lambda) = 0$, where the sum $\sum_\nu d_\lambda^\nu \PD(\sigma_\nu)$ is nonzero in $H^*(\G(k,N);R)$. In this sense, $\nabla$ is not well-defined for every partition $\lambda$ with at most $k$ parts. If we define $\nabla$ on $H^*(\G(k,N);R)$ by the given formula just on the basis $\{\PD(\sigma_\lambda)\}_\lambda$ where $\lambda_1 \leq N - k$, then $\nabla$ will be well-defined, but it will not satisfy the Leibniz rule. The main content of Corollary~\ref{cor:DdefinedOnCohOfGkN} is the claim that $d_\lambda^\nu$ is divisible by $N$ whenever $\lambda_1 > N - k$ and $\nu_1 \leq N - k$. We verify this directly by providing an explicit formula for $d_\lambda^\nu$. 

\begin{prop}\label{prop:formulaforDcoefficientsInSchurPolys}
	The coefficients $d_\lambda^\nu \in \Z$ defined by the equation $\nabla(s_\lambda) = \sum_\nu d_\lambda^\nu s_\nu$ are given by \[
		d_\lambda^\nu = \begin{cases}
			\lambda_\ell + k - \ell & \nu \text{ is obtained from }\lambda\text{ by decreasing }\lambda_\ell\text{ by }1\\
			0 & \text{else.}
		\end{cases}
	\]
\end{prop}
\noindent Thus, if $\lambda_1 = N - k + 1$ and $\nu$ is obtained from $\lambda$ by decreasing $\lambda_1$ by $1$, then $d_\lambda^\nu = N$. 
\begin{proof}
	By the definition of the Schur polynomial $s_\lambda$, we have the equality \[
		s_\lambda \cdot\prod_{i < j} (X_i - X_j) = \det(X_i^{\lambda_j + k - j}).
	\]Since $\nabla(X_i - X_j) = 0$, it follows from the Leibniz rule that \begin{align*}
		\nabla(s_\lambda) \cdot \prod_{i < j} (X_i - X_j) &= \nabla(\det(X_i^{\lambda_j + k - j})) = \sum_{k=1}^k \det(M_\ell)
	\end{align*}where $M_\ell$ is obtained from the matrix $X_i^{\lambda_j + k - j}$ by applying $\nabla$ to each entry in the $\ell$th row. If decreasing $\lambda_\ell$ by $1$ does not yield a valid partition, which occurs precisely when $\lambda_\ell = \lambda_{\ell + 1}$, then $\det(M_\ell) = 0$ because the $\ell$th row of $M_\ell$ is just $\lambda_\ell + k - \ell$ times the $(\ell + 1)$th row of $M_\ell$. If decreasing $\lambda_\ell$ by $1$ does yield a valid partition $\nu$, then \[
		\det(M_\ell) = (\lambda_\ell + k - \ell)\det(X_i^{\nu_j + k - j}).
	\]The result now follows by dividing both sides by $\prod_{i < j} (X_i - X_j)$.
\end{proof}

\begin{rem}
	Proposition~\ref{prop:formulaforDcoefficientsInSchurPolys} is a special case of \cite[Proposition 1.1]{MR4098997}, which computes the action of $\nabla$ on Schubert polynomials, which generalize Schur polynomials. The author thanks Christian Gaetz for this reference. 
\end{rem}

\raggedright
\bibliography{diffOpKR}
\bibliographystyle{alpha}

\vspace{10pt}

\textit{Department of Mathematics}

\textit{Harvard University}

\textit{Science Center, 1 Oxford Street}

\textit{Cambridge, MA 02138}

\textit{USA}

\vspace{10pt}

\textit{Email:} \texttt{jxwang@math.harvard.edu}

\end{document}

%% file: paperstyle.tex
\usepackage[english]{babel} 
\usepackage[margin=1.5in]{geometry} 
\usepackage{amsmath, amsfonts, amssymb, amsthm} 
\usepackage[hidelinks]{hyperref} 
\usepackage[scr=boondoxo,cal=euler]{mathalfa} 
\usepackage{tikz-cd} 
\usepackage{pxfonts} 

\usepackage{indentfirst} 
\usepackage{fancyhdr} 
\usepackage{graphicx} 
\usepackage{enumitem} 
\usepackage{microtype} 
\usepackage{pdfpages} 
\usepackage{centernot} 
\usepackage{ragged2e} 
\usepackage{mathabx}
\usepackage{pinlabel}
\usepackage{caption}
\usepackage{etoolbox}

\allowdisplaybreaks 

\usetikzlibrary{decorations.pathmorphing}

\theoremstyle{plain}
\newtheorem{thm}{Theorem}[section]
\newtheorem*{thm*}{Theorem}
\newtheorem{lem}[thm]{Lemma}
\newtheorem*{lem*}{Lemma}
\newtheorem{cor}[thm]{Corollary}
\newtheorem*{cor*}{Corollary}
\newtheorem{prop}[thm]{Proposition}
\newtheorem*{prop*}{Proposition}

\newtheorem*{ques*}{Question}

\theoremstyle{definition}
\newtheorem{df}[thm]{Definition}
\newtheorem*{df*}{Definition}
\newtheorem*{dfs*}{Definitions}

\newtheorem*{exercise*}{Exercise}

\theoremstyle{remark}
\newtheorem{rem}[thm]{Remark}
\newtheorem*{rem*}{Remark}

\input{commands.tex}

%% file: commands.tex
\newcommand{\RFoam}{R\text{-}\bo\Foam}
\newcommand{\ZFoam}{\Z\text{-}\bo\Foam}
\newcommand{\RMod}{R\text{-}\bo\Mod}

\newcommand{\qbinom}[2]{\genfrac{[}{]}{0pt}{}{#1}{#2}}

\DeclareMathOperator{\KR}{KR}
\DeclareMathOperator{\Mod}{Mod}
\DeclareMathOperator{\KRC}{KRC}
\DeclareMathOperator{\PD}{PD}
\DeclareMathOperator{\Foam}{Foam}

\newcommand{\fk}[1]{\mathfrak{#1}}

\newcommand{\sr}[1]{\mathscr{#1}}
\newcommand{\bo}[1]{\boldsymbol{#1}} 

\newcommand{\Z}{\mathbf{Z}} 
\newcommand{\Q}{\mathbf{Q}} 
\newcommand{\R}{\mathbf{R}} 
\newcommand{\C}{\mathbf{C}} 
\renewcommand{\P}{\mathbf{P}} 
\newcommand{\G}{\mathbf{G}} 

\newcommand{\x}{\times}

\newcommand{\emp}{\emptyset}

\renewcommand{\sl}{\fk{sl}}

\newcommand{\rKh}{\smash{\overline{\Kh}}}
\newcommand{\rKR}{\smash{\overline{\KR}}}

\DeclareMathOperator{\rk}{rk}

\DeclareMathOperator{\Hom}{Hom}

\DeclareMathOperator{\Id}{Id}

\DeclareMathOperator{\Kh}{Kh}

\patchcmd{\thmhead}{(#3)}{#3}{}{}

\makeatletter
\g@addto@macro\bfseries{\boldmath}
\makeatother